\documentclass[11pt,reqno]{amsproc}

\usepackage[sc]{mathpazo}\renewcommand{\mathbf}{\mathbold}
\normalfont
\usepackage[T1]{fontenc}

\usepackage{misc/packages}
\newcommand{\A}{\mathbb{A}}

\newcommand{\C}{\mathbb{C}}
\newcommand{\Z}{\mathbb{Z}}
\newcommand{\N}{\mathbb{N}}
\newcommand{\Q}{\mathbb{Q}}
\newcommand{\R}{\mathbb{R}}

\newcommand{\bA}{\mathbb{A}}
\newcommand{\bB}{\mathbb{B}}

\newcommand{\bC}{\mathbb{C}}
\newcommand{\bM}{\mathbb{M}}

\newcommand{\bI}{\mathbb{I}}

\newcommand{\cA}{\mathcal{A}}

\newcommand{\cP}{\mathcal{P}}
\newcommand{\cB}{\mathcal{B}}
\newcommand{\cC}{\mathcal{C}}

\newcommand{\cM}{\mathcal{M}}

\newcommand{\apl}{\mathcal{A}_{PL}}
\newcommand{\cN}{\mathcal{N}}

\newcommand{\sTop}{\mathsf{Top}}

\newcommand{\sSet}{\mathsf{sSet}}
\newcommand{\op}{\mathsf{op}}

\newcommand{\Hom}{\mathsf{Hom}}

\newcommand{\sVec}{\mathsf{Vec}_\Q}
\newcommand{\Ker}{\mathrm{Ker} \,}
\newcommand{\im}{\mathrm{Im} \,}
\newcommand{\Coker}{\mathrm{Coker}\:}
\newcommand{\Coim}{\mathrm{Coim}\:}
\newcommand{\Cone}{\mathrm{Cone}}
\newcommand{\Ho}{\mathrm{Ho}}

\newcommand{\bS}{\mathbf{S}}

\newcommand{\mA}{\mathbbm{A}}

\newcommand{\id}{\mathsf{id}}
\newcommand{\Ch}{\mathsf{Ch}^*}
\newcommand{\CDGA}{\mathsf{CGDA}}

\newcommand{\Di}{\mathbb{D}}
\newcommand{\Fun}{\mathsf{Fun}}
\newcommand{\I}{\mathbb{I}}

\newcommand{\Sp}{\mathbb{S}}

\newcommand{\W}{\mathbb{W}}
\newcommand{\X}{\mathbb{X}}

\newcommand{\Y}{\mathbb{Y}}
\newcommand{\V}{\mathbb{V}}

\newcommand{\colim}{\mathsf{colim}}

\newcommand{\DDelta}{\mathbf{\Delta}}

\newcommand{\weak}{\overset{\sim}{\to}}

\newtheorem*{theorem*}{Theorem}
\newtheorem{theorem}{Theorem}[section]

\newtheorem{lemma}[theorem]{Lemma}
\newtheorem{corollary}[theorem]{Corollary}

\newtheorem{proposition}[theorem]{Proposition}

\theoremstyle{definition}
\newtheorem{definition}[theorem]{Definition}

\theoremstyle{remark}
\newtheorem{remark}[theorem]{Remark}
\newtheorem{notation}[theorem]{Notation}

\numberwithin{equation}{section}

\newcommand{\kathryn}[1]{{\color{teal}[[\textbf{Kathryn says: }#1]]}}

\title{Cell decompositions of persistent minimal models}
\date{\today}
\author{Kathryn Hess, Samuel Lavenir, Kelly Maggs}

\begin{document}

\maketitle
% \vspace{-32pt}
\begin{center}
\begin{small}
    \today
\end{small}
\end{center}

\maketitle

\begin{abstract}
    In this article we generalize the main structure theorems of rational homotopy theory to the persistent setting. Our main motivation is the computation of an explicit finite, cellular presentation of the persistent minimal model that completely characterizes the rational homotopy type of copersistent simply-connected spaces. We achieve this via an explicit construction of the minimal model of a tame persistent CDGA as an iterated sequence of cell attachments. As an application of our results, we construct an explicit decomposition of the rational Postnikov tower of simply-connected copersistent spaces in terms of a tower of persistent Eilenberg-Maclane intervals.
\end{abstract}

\tableofcontents

\section{Introduction}

Persistence theory studies how algebraic invariants vary along a parametrized family of objects in a category \cite{Oudot_2015,Carlsson05}. Parametrized objects are encoded as functors $F: \cP \to \cC$ from a poset $\cP$ into a category $\cC$. Post-composing with the functorial invariant of interest yields the associated persistent invariant. The simplest but most ubiquitous example are functors $[n]\to \mathsf{Vec}_{\mathbf{k}}$ from the standard poset $([n], \leq)$ over $[n] = \{ 0,1,\ldots, n\}$ to vector spaces $\mathsf{Vec}_\mathbf{k}$ over a field $\mathbf{k}$ which form the category of (finitely-indexed) \textit{persistence modules} and have been studied extensively. An interval module $\bI_{[s_\gamma,t_\gamma)}$ represents a functor that is constant at $\Q$ over the half-open interval $[s_\gamma,t_\gamma) \subset [n]$ and zero elsewhere. Finitely-indexed persistence modules admit an extremely useful decomposition, namely, the \textit{interval decomposition} $\V \cong \bigoplus_\gamma \bI_{[s_\gamma,t_\gamma)}$ \cite{Webb_1985,Carlsson05,chazal2016structure,Crawley-Boevey_2015}. In such cases, the persistent module is completely described (up to isomorphism) by the set of pairs $(s_\gamma, t_\gamma)$ of birth-times and death-times. 
 
 The goal of this article is to use the interval decomposition to generate novel structure theorems at the homotopical level. Given that persistence-related decompositions exist only over fields, rational homotopy theory is a natural candidate for studying persistent homotopical behaviour. In this paper, we thus focus on the persistent versions of model categories associated with  rational homotopy theory -- commutative differential graded algebras (CDGAs) and simplicial sets. Rather than direct sums, persistent CDGAs decompose into a minimal cell complex consisting of persistent cells and attaching maps, i.e., a persistent version of the classical Sullivan minimal model. Prior to this work, however, this construction has yet to be fully integrated with the results of persistence theory.

The main technical obstruction is that replacing a CDGA of polynomial differential forms on a space with a Sullivan minimal CDGA is \textit{homotopy functorial} (in the sense of Diagram \ref{homotopy-functoriality-sullivan-rep}) rather than functorial on the nose. This requires particular care in the persistent case as the structure maps between algebras of polynomial differential forms are part of the data. Indeed, we define a $k$-minimal model for a a persistent CDGA $\bA : [n] \to \CDGA$ to be a diagram
   % https://q.uiver.app/#q=WzAsMTAsWzIsMCwiXFxtTShyKSJdLFs0LDAsIlxcbU0ocisxKSJdLFsyLDIsIlxcbUEocikiXSxbNCwyLCJcXG1BKHIrMSkiXSxbMCwwLCJcXGxkb3RzIl0sWzYsMCwiXFxtTShyKzIpIl0sWzYsMiwiXFxtQShyKzIpIl0sWzAsMiwiXFxsZG90cyJdLFs4LDAsIlxcbGRvdHMiXSxbOCwyLCJcXGxkb3RzIl0sWzAsMV0sWzAsMiwibShyKSJdLFsyLDNdLFsxLDMsIm0ocisxKSJdLFs0LDBdLFsxLDVdLFsxLDIsIkgocikiLDAseyJzaG9ydGVuIjp7InNvdXJjZSI6MjAsInRhcmdldCI6MjB9LCJsZXZlbCI6Mn1dLFs1LDYsIm0ocisyKSJdLFszLDZdLFs1LDMsIkgocisxKSIsMCx7InNob3J0ZW4iOnsic291cmNlIjoyMCwidGFyZ2V0IjoyMH0sImxldmVsIjoyfV0sWzcsMl0sWzUsOF0sWzYsOV1d
\begin{equation} \label{tame-minimal-model} \begin{tikzcd}[sep=scriptsize]
	\ldots && {\bM(r)} && {\bM(r+1)} && {\bM(r+2)} && \ldots \\
	\\
	\ldots && {\bA(r)} && {\bA(r+1)} && {\bA(r+2)} && \ldots
	\arrow[from=1-1, to=1-3]
	\arrow[from=1-3, to=1-5]
	\arrow["{m(r)}", from=1-3, to=3-3]
	\arrow[from=1-5, to=1-7]
	\arrow["{H(r)}", shorten <=13pt, shorten >=13pt, Rightarrow, from=1-5, to=3-3]
	\arrow["{m(r+1)}", from=1-5, to=3-5]
	\arrow[from=1-7, to=1-9]
	\arrow["{H(r+1)}", shorten <=14pt, shorten >=14pt, Rightarrow, from=1-7, to=3-5]
	\arrow["{m(r+2)}", from=1-7, to=3-7]
	\arrow[from=3-1, to=3-3]
	\arrow[from=3-3, to=3-5]
	\arrow[from=3-5, to=3-7]
	\arrow[from=3-7, to=3-9]
\end{tikzcd}\end{equation} such that each $m(r) : \bM(r) \to \A(r)$ is a $k$-minimal model, abbreviating such a diagram to the notation $\bM \to_2 \bA$. Explicit homotopies are part of the data, and need only be defined between neighbouring indices.

The main tool we introduce is an algebraic notion of persistent cell attachment (\ref{persistent-hirsch-extension}), which we use to formally adjoin free, graded-commutative intervals to persistent cochain complexes and commutative differential graded algebras (CDGAs) --- a procedure we call \textit{interval surgery}. Attachment of an interval $\bI_{[p,q)}^k$ in degree $k$ is performed via push outs 
\begin{equation}
    \begin{tikzcd}
        \Lambda \Sp^{k+1}_{[p,q)} \ar[r, "\tau"] \ar[d] & \bA \ar[d] \\
        \Lambda \Di^{k+1}_p \ar[r] & \bA \otimes_\varphi \Lambda \bI^k_{[p,q)} \arrow[ul, phantom, "\ulcorner", very near start]
    \end{tikzcd}
    \end{equation} with \textit{interval spheres} $\Lambda \Sp_{[p,q)}^k$ \textit{and disks} $\Lambda \Di^{k+1}_p$ providing an algebraic model of persistent spheres and disks. Our main result is the inductive construction of a persistent Sullivan minimal model via a series of explicit algebraic cell attachments in the category of persistent CDGAs.

\begin{theorem*}[Informal pCDGA Structure Theorem]\label{informal-structure-theorem}
    Every simply-connected persistent CDGA $\bA$ admits a persistent minimal model $\bM$ whose skeletal filtration is constructed as a sequence of degree $k$ persistent cell attachments
    % https://q.uiver.app/#q=WzAsNSxbMCwwLCJcXExhbWJkYSBcXGJTXntrKzF9KEhea1xcQ19tKSJdLFswLDIsIlxcTGFtYmRhIFxcYkRee2srMX0oSF5rXFxDX20pIl0sWzIsMCwiXFxiTV97ay0xfSJdLFsyLDIsIlxcYk1fayJdLFs0LDAsIlxcYkEiXSxbMCwxLCIiLDAseyJzdHlsZSI6eyJ0YWlsIjp7Im5hbWUiOiJob29rIiwic2lkZSI6InRvcCJ9fX1dLFswLDJdLFsxLDNdLFsyLDMsIiIsMix7InN0eWxlIjp7InRhaWwiOnsibmFtZSI6Imhvb2siLCJzaWRlIjoidG9wIn19fV0sWzIsNCwibSJdLFszLDQsIm1fayIsMix7ImN1cnZlIjozfV0sWzIsNCwiMiIsMix7ImxhYmVsX3Bvc2l0aW9uIjo4MH1dLFszLDQsIjIiLDIseyJsYWJlbF9wb3NpdGlvbiI6OTAsImN1cnZlIjozfV1d
\begin{equation} \label{informal-cofiber-sequence-diagram} \begin{tikzcd}[sep=scriptsize]
	{\bigotimes_i \Lambda \Sp^{k+1}_{[p_i,q_i)}} && {\bM_{k-1}} && \bA \\
	\\
	{\bigotimes_i \Lambda \Di^{k+1}_{p_i}} && {\bM_k}
	\arrow[from=1-1, to=1-3]
	\arrow[hook, from=1-1, to=3-1]
	\arrow["m_{k-1}", from=1-3, to=1-5]
	\arrow["2"'{pos=0.8}, from=1-3, to=1-5]
	\arrow[hook, from=1-3, to=3-3]
	\arrow[from=3-1, to=3-3]
	\arrow["{m_k}"', from=3-3, to=1-5]
	\arrow["2"'{pos=0.9}, from=3-3, to=1-5]
\end{tikzcd}\end{equation}
\end{theorem*} \noindent This approach allows us to produce explicit formulae involving homotopies for our construction that we believe will form the basis of practical algorithms in future work. 

The above cell decomposition also has an interpretation at the level of spaces. In classical rational homotopy theory, the algebraic $k$-cells in the minimal model bijectively correspond to the generators of the $k$-th rational homotopy groups of the space and make explicit computations possible. The basis of the correspondance is that there exists a spatial realization functor $\langle - \rangle : \CDGA \to \sSet$ taking the skeletal filtration of the Sullivan minimal model to the rational Postnikov tower to the space it models. An application of our inductive construction of the persistent minimal model allows us to generalize this to the persistent level. 
\begin{theorem*}[Copersistent Postnikov tower decomposition] \label{persistent-eckmann-hilton}  Let $\X : [n]^{op} \to \sSet$ be a copersistent, simply-connected space $\X$ with minimal model $\bM \to_2 \mathcal{A}_{PL}(X)$. The image of the skeletal filtration \ref{informal-cofiber-sequence-diagram} of $\bM$ under spatial realization 
    % https://q.uiver.app/#q=WzAsOSxbMCwwLCJcXGJpZ290aW1lcyBcXExhbWJkYSBcXG1hdGhiYntTfV57aysxfV97W3MsdCl9Il0sWzAsMiwiXFxiaWdvdGltZXMgXFxMYW1iZGEgXFxtYXRoYmJ7RH1ee2srMX1fe1tzLFxcaW5mdHkpfSJdLFsyLDAsIlxcbWF0aGJie019X3trLTF9Il0sWzIsMiwiXFxtYXRoYmJ7TX1fe2t9Il0sWzMsMSwiXFx4cmlnaHRhcnJvd3tcXGxhbmdsZSAtIFxccmFuZ2xlfSJdLFs0LDAsIlxcbGFuZ2xlIFxcbWF0aGJie019X3trfVxccmFuZ2xlIl0sWzYsMiwiSyhcXHBpX2teXFxRKFxcbWF0aGJie1h9KSwgaysxKSJdLFs2LDAsIlBLKFxccGlfa15cXFEoXFxtYXRoYmJ7WH0pLGsrMSkiXSxbNCwyLCJcXGxhbmdsZSBcXG1hdGhiYntNfV97ay0xfVxccmFuZ2xlIl0sWzAsMSwiIiwyLHsic3R5bGUiOnsidGFpbCI6eyJuYW1lIjoibW9ubyJ9fX1dLFswLDIsIlxcb3RpbWVzXFx0YXUiXSxbMSwzXSxbMiwzLCIiLDAseyJzdHlsZSI6eyJ0YWlsIjp7Im5hbWUiOiJtb25vIn19fV0sWzMsMCwiIiwxLHsic3R5bGUiOnsibmFtZSI6ImNvcm5lciJ9fV0sWzcsNiwiIiwyLHsic3R5bGUiOnsiaGVhZCI6eyJuYW1lIjoiZXBpIn19fV0sWzUsOCwiIiwwLHsic3R5bGUiOnsiaGVhZCI6eyJuYW1lIjoiZXBpIn19fV0sWzgsNiwiXFxwcm9kXFxsYW5nbGUgXFx0YXUgXFxyYW5nbGUiXSxbNSw3XSxbNSw2LCIiLDEseyJzdHlsZSI6eyJuYW1lIjoiY29ybmVyIn19XV0=
\[\begin{tikzcd}[sep=tiny]
	{\bigotimes \Lambda \mathbb{S}^{k+1}_{[s,t)}} && {\mathbb{M}_{k-1}} && {\langle \mathbb{M}_{k}\rangle} && {\prod PK(\bI_{[s,t)}^\vee,k+1)} \\
	&&& {\xrightarrow{\langle - \rangle}} \\
	{\bigotimes \Lambda \mathbb{D}^{k+1}_{[s,\infty)}} && {\mathbb{M}_{k}} && {\langle \mathbb{M}_{k-1}\rangle} && { \prod K(\bI_{[s,t)}^\vee,k+1)}
	\arrow["\otimes\tau", from=1-1, to=1-3]
	\arrow[tail, from=1-1, to=3-1]
	\arrow[tail, from=1-3, to=3-3]
	\arrow[from=1-5, to=1-7]
	\arrow[two heads, from=1-5, to=3-5]
	\arrow["\lrcorner"{anchor=center, pos=0.125}, draw=none, from=1-5, to=3-7]
	\arrow[two heads, from=1-7, to=3-7]
	\arrow[from=3-1, to=3-3]
	\arrow["\lrcorner"{anchor=center, pos=0.125, rotate=180}, draw=none, from=3-3, to=1-1]
	\arrow["{\prod\langle \tau \rangle}", from=3-5, to=3-7]
\end{tikzcd}\] recovers a copersistent Postnikov tower for $\X$.
\end{theorem*} More explicitly, the $K(\bI_{[s_\gamma,t_\gamma)}^\vee,k)$ are Eileberg-Maclane spaces $K(\Q,k)$ that are constant over the contravariant invariant interval $[s_\gamma,t_\gamma)$ and zero elsewhere. Given that $\lim \langle \bM_k \rangle $ is rationally homotopy equivalent to $\mathbb{X}$, the construction on the right is the 'copersistent Postnikov tower' for (the rational homotopy type of) $\X$. This reflects a basic correspondence:\\ 
$$\Big\{ \parbox{10em}{\centering attachments of\\ interval spheres} \Big\} \hspace{2em} \leftrightarrow \hspace{2em} \Big\{ \parbox{15em}{ \centering co-attachments of co-persistent\\ Eilenberg-Mac Lane intervals} \Big\} \vspace{1em}.$$ If all attaching maps in the persistent minimal model are trivial, then
$$\X \simeq \langle \bM \rangle \simeq \prod_k K( \pi_k^\Q(\X), k) \simeq \prod_\gamma K(\bI_{[s_\gamma,t_\gamma)},k_\gamma),$$ which we interpret as a `homotopical interval decomposition' of $\X$ at the level of copersistent spaces describing the persistent rational homotopy type.

The interval decomposition offers a useful presentation of the persistent homology type in terms of birth and death times of intervals. Analogously, the explicit construction of the minimal model produces a finite \textit{presentation} of the persistent rational homotopy type. Each interval $\bI^\vee_{[p,q)}$ in the persistent rational homotopy group provides a generator $\gamma_p \in \bM$ for the minimal model which is attached into the complex in two positions: (1) at birth-time, it is attached to a cocycle $d \gamma_p = v_p$ to define the coboundary and (2) at death-time, it is attached to either $0$ or a non-trivial product $$\gamma_q = u_q = u_q^{(1)} \wedge \ldots \wedge u_q^{(j)}$$ of lower-degree generators. In short, $\bM$ has the presentation:
% https://q.uiver.app/#q=WzAsNixbMCwwLCJcXGJNID0gXFxMYW1iZGEoXFxtYXRoc2Nye0d9IFxcbWlkIFxcbWF0aHNjcntSfSkiXSxbMiwwLCJcXEJpZ1xceyBcXCwgXFwsIFxcZ2FtbWFfcCBcXGluIFxcIFxcYk1eayhwKVxcLCBcXCwgXFxCaWdcXH0iXSxbMiwxLCJcXHRleHR7R2VuZXJhdG9ycyB9IFxcbWF0aHNjcntHfSJdLFs0LDAsIlxcQmlnZ1xceyBcXCwgXFwsIGRcXGdhbW1hX3AgPSB2X3AsIFxcZ2FtbWFfcSA9IHVfcSBcXCwgXFwsIFxcQmlnZ1xcfSJdLFs0LDEsIlxcdGV4dHtSZWxhdGlvbnMgfSBcXG1hdGhzY3J7Un0iXSxbMCwxLCJcXHRleHR7UHJlc2VudGF0aW9ufSJdXQ==
\[\begin{tikzcd}[row sep=-0.5em]
	{\bM = p\Lambda(\mathscr{G} \mid \mathscr{R})} && {\Big\{ \, \, \gamma_p \in \ \bM^k(p)\, \, \Big\}} && {\Bigg\{ \, \, d\gamma_p = v_p, \gamma_q = u_q \, \, \Bigg\}} \\
	{\text{Presentation}} && {\text{Generators } \mathscr{G}} && {\text{Relations } \mathscr{R}}
\end{tikzcd}\]
To conclude the paper, we use this structure to describe the algebraic constraints on the possible rational homotopy types corresponding to a given persistent rational homotopy group arising from the naturality of coboundary operator.

\subsection*{Related work} 
It has earlier been observed that the persistent rational homotopy groups admit an interval decomposition under certain tameness assumptions \cite{Bubenik_de_Silva_Scott_2014}. The nascent field of (non-rational) persistent homotopy theory has been studied in \cite{Blumberg_Lesnick_2023,Lanari_Scoccola_2023,jardine2019data, jardine2020persistent}. In particular, the work of \cite{memoli2022persistent} seems to be the first to address persistent rational homotopy theory, doing so in the context of Rips complexes of metric spaces. Indeed, a persistent minimal model for rational homotopy theory was introduced in the concurrent work of \cite{zhou2023persistent}, which shares a similar definition. Here we do not address questions related to stability and interleaving, which are central to that work. The structural aspects of the persistent minimal model we study here should provide complementary insights, useful in addressing such questions.

\subsection*{Acknowledgements} The authors thank Jérôme Scherer for his helpful, detailed feedback on earlier versions of this manuscript and for valuable discussions about Postnikov towers and Eckmann-Hilton duality. K.M. was supported by the European Union’s Horizon 2020 Research and Innovation Program under Marie Skłodowska-Curie Grant Agreement No 859860. S.L. was supported by Swiss National Science Foundation, grant/award number: 200020 18858. 

\begin{comment}
For a finite sequence of finite simplicial complexes
$ S_0 \xrightarrow{f_0} \ldots \xrightarrow{f_{n-1}} S_n,$
applying the rational homotopy group functor $\pi_k( \, \cdot \, ) \otimes \Q$ for any $k \geq 0$ to induces a finite sequences of rational vector spaces
\begin{equation} \label{rational_persistence_module}
\pi_k(S_0) \otimes \Q \xrightarrow{(f_0)_*} \ldots \xrightarrow{(f_{n-1})_*} \pi_k(S_n) \otimes \Q.
\end{equation}
The primary goal of this document is provide a roadmap to compute the persistence diagram of \ref{rational_persistence_module} in the case that each $S_i$ is simply connected.\\

\textbf{Related Work}
\end{comment}

\subsection{Notation and conventions}

\begin{itemize}
    \item We apply cohomological conventions, i.e., differentials increase degree.
    \item We use the word \textit{space} as a synonym for \textit{simplicial set}.
    \item All vector spaces are over the field $\Q$ of rationals.
    \item  We write $V^\vee$ for the dual vector space $\Hom_\Q(V, \Q)$.
    \item We denote by $\Ch_\Q$ the category of non-negatively graded cochain complexes over $\Q$ and by $\CDGA_\Q$ the category of commutative differential graded algebras (CDGAs) over $\Q$, i.e., the category of commutative monoids in $\Ch_\Q$.
    \item For $(C,d) \in \Ch_\Q$, the shifted complex $(C[n],d')$ is defined by $C[n]^k=C^{k-n}$ with differentials $d'_k=(-1)^n d_k$.
    \item A CDGA $\mathcal{A}$ is said to be connected if $H^0(\mathcal{A})=\Q$ and simply-connected if $H^1(\mathcal{A}) = 0$ additionally. 
    \item If $\mathcal{A} \in \CDGA_\Q$ is connected and therefore admits an augmentation $\cA \to \Q$, then 
    \begin{itemize} 
    \item $\overline{\cA}=\bigoplus_{n>0}\cA^n$ is the kernel of the augmentation;
    \item $Q(\cA)$ is the complex of \textit{indecomposables}
    $$Q(\cA)=\Coker(\overline{\cA}\otimes \overline{\cA} \to \overline{\cA}).$$
    \end{itemize}

    \item We use the term \textit{finite type} in three settings.
    \begin{itemize}
        \item A simplicial set is finite type if it has finitely many non-degenerate simplices.
        \item A cochain complex is finite type if it is degree-wise finite dimensional.
        \item A CDGA is finite type if its cohomology is degree-wise finite dimensional.
    \end{itemize}
\end{itemize}

\section{Background}

We first recall the basic properties of minimal models and their role in rational homotopy theory. The main reference for this section is \cite{griffithsmorgan}.

\subsection{Cones and homotopies}  \subsubsection{CDGA homotopies} Let $\Lambda(t,dt)$ be the free CDGA representing the polynomial differential forms on the interval, where the generator $t$ is of degree 0. This algebra admits two natural augmentations 
$ \varepsilon_0, \varepsilon_1 : \Lambda(t,dt) \to \Q $ defined by $\varepsilon_0(t)=0$ and $\varepsilon_1(t)=1$, corresponding to the restrictions to the start and end of the interval. Note that for degree reasons $\varepsilon_0(dt) = \varepsilon_1(dt)= 0$. A \textit{homotopy} between CDGA maps $f,g : \cA \to \cB$ is an algebra homomorphism $H : \cA \to \cB \otimes \Lambda(t,dt)$ such that $(id \cdot \epsilon_0) H = f$ and $(id \cdot \epsilon_0) H = g$. 

\subsubsection{Cochain complexes cones} Let $f : A \to B$ be a cochain map. We denote by $C_f$ the mapping cone of $f$, which has components $C_f^n = A^{n+1} \oplus B^{n}$ and differential $d:C_f^n \to C_f^{n+1}$ given by $$
d(a,b)=(da, f(a)-db).
$$
There is a natural map $B\to C_f$ given by $b\mapsto (0,-b)$. The sequence $A \xrightarrow{f} B \xrightarrow{} C_f$ induces a long exact sequence in cohomology $$
 \cdots \to H^{n-1}C_f \xrightarrow{} H^nA \xrightarrow{f^*} H^n B \xrightarrow{} H^nC_f \xrightarrow{} H^{n+1}A \to \cdots 
$$
in which the connecting homomorphisms are induced by the projection map $C_f \to A[1]$ given by $(a,b)\mapsto a$. Given a homotopy commutative square of cochain complexes $$
\begin{tikzcd} 
    A \ar[r, "u"]  \ar[d, "f"'] & A' \ar[dl, Rightarrow, "h", shorten = 1.0ex] \ar[d, "f'"]  \\ 
    B \ar[r, "v"'] & B'  
\end{tikzcd} 
$$
there is an induced cochain map on mapping cones $\varphi:C_f \to C_{f'}$ given by the formula $$
\varphi(a,b)=(u(a), v(b)+h(a)).
$$
\subsubsection{Integrating homotopies} A homotopy in CDGA can be integrated fiber-wise to a homotopy in $\Ch(\Q)$. The assignments
$$
t^k\mapsto 0
\:\:\:\:\:\:\:\:\:\:\:\:\:\:\:
t^kdt\mapsto \frac{1}{k+1}t^{k+1}
\:\:\:\:\:\:\:\:\:\:\:\:\:\:\:
\forall k \geq 0
$$
define a linear map $\int_0^t:\Lambda(t,dt)\to \Lambda(t,dt)$ of degree $-1$. We denote by $\int_0^1$ the composite $\varepsilon_1 \circ \int_0^t:\Lambda(t,dt) \to \Lambda(t,dt) \to \Q$. If $H:f\sim g$ is a homotopy between CDGA homomorphisms $f,g: \cA \to \cB$, the composite $$
 \cA \xrightarrow{H} \cB \otimes \Lambda(t,dt) \xrightarrow{\id \otimes \int_0^1} \cB \otimes \Q \cong \cB
$$
is a linear map $\int_0^1H:\cA \to \cB$ of degree $-1$. A straightforward calculation shows that it satisfies the identity
$$
d \int_0^1 H(a) + \int_0^1 H(da) = g(a) - f(a)
$$
for every $a\in \cA$. In other words, $\int_0^1H$ is a cochain homotopy between the maps $f$ and $g$. We say that $\int_0^1 H$ is the cochain homotopy \textit{underlying} the CDGA homotopy $H$.
\begin{comment} Moreover, given a homotopy commutative diagram
$$
\begin{tikzcd} 
    A \ar[r, "u"]  \ar[d, "f"'] & A' \ar[dl, Rightarrow, "h", shorten = 1.0ex] \ar[d, "f'"]\ar[r, "u'"] & A' \ar[d, "f''"]\ar[dl, Rightarrow, "h'", shorten = 1.0ex] \\ 
    B \ar[r, "v"'] & B' \ar[r, "v'"']  & B'' 
\end{tikzcd} 
$$
the composite $C_f \to C_{f'} \to C_{f''}$ is exactly the map induced on mapping cones by the diagram 
$$
\begin{tikzcd} 
    A \ar[rr, "u'u"]  \ar[d, "f"'] && A' \ar[dll, Rightarrow, "k", shorten = 2.0ex] \ar[d, "f''"]  \\ 
    B \ar[rr, "v'v"'] && B'  
\end{tikzcd} 
$$
in which the homotopy $k$ is defined by $v'h + h'u$. \end{comment}

\begin{remark} \label{cochain-cone} 
    We will often speak of the mapping cone $C_f$ of a CDGA homomorphism $f:\cA \to \cB$. By this we mean the mapping cone of the \textit{underlying cochain map}. In this situation, it should be noted that $C_f$ \textit{does not} carry any natural algebra structure. In particular, the notation $C_f$ should not be confused with the mapping cone in CDGA, which we will never use in this paper. In a similar spirit, when facing a homotopy commutative square of CDGAs $$
    \begin{tikzcd} 
        \cA \ar[r, "u"]  \ar[d, "f"'] & \cA' \ar[dl, Rightarrow, "H", shorten = 1.0ex] \ar[d, "f'"]  \\ 
        \cB \ar[r, "v"'] & \cB'  
    \end{tikzcd} 
    $$
    the induced map on cones $\varphi:C_f\to C_{f'}$ is merely a cochain map, and is given by the formula $$
    \varphi(a,b)=\big(\:u(a) \:,\: v(b) + \int_0^1H(a) \:\big)
    $$
    where we have used the \textit{cochain homotopy} $\int_0^1H$ which underlies the CDGA homotopy $H$.

\subsection{Minimal models} In this section we briefly review the theory of Sullivan minimal models.

\subsubsection{Cochain spheres and disks} Each vector space $V$ has associated spheres $S^k(V)$ and disks $D^k(V)$ for every integer $k\geq 0$. The complex $S^k(V)$ consists of a copy of $V$ in degree $k$. Similarly, the complex $D^k(V)$ consists of a copy of $V$ in each of degrees $(k-1)$ and $k$, where the only non zero differential is the identity. When $k=0$, we adopt the convention that $D^0(V)=0$. We also write $S^k=S^k(\Q)$ and $D^k=D^k(\Q)$. Given a graded $\Q$-vector space $V=\bigoplus_{k\geq 0} V_k$, we define an associated sphere $S(V)$ and disk $D(V)$. These are cochain complexes are defined by $$
S(V)=\bigoplus_{k\geq 0} S^k(V_k) 
\:\:\:\:\:\:\:,\:\:\:\:\:\:\:
D(V)=\bigoplus_{k> 0} D^k(V_k).
$$
Note that there are evident inclusions $S(V) \subseteq D(V)$ and that $D(V)/S(V)=S(V[-1])=S(V)[-1]$ for any choice of graded vector space $V$.

\subsubsection{Hirsch extensions} Cell attachments along generating cofibrations $\Lambda S^k \to \Lambda D^k$ in CDGA and their tensor products play an important role in rational homotopy theory. A \textit{Hirsch extension} of a CDGA $\cA$ is an algebra homomorphism $\cA \to \cA\langle V \rangle$ that fits into a pushout square of CGDA maps 
    $$
    \begin{tikzcd}
        \Lambda S(V) \ar[r]\ar[d] & \cA \ar[d] \\
        \Lambda D(V) \ar[r] & \cA \langle V \rangle \arrow[ul, phantom, "\ulcorner", very near start]
    \end{tikzcd}
    $$
    for some choice of graded $\Q$-vector space $V$, and is \textit{of degree} $k$ when $V$ is concentrated in degree $(k+1)$. A Hirsch extension $\cA \subseteq \cA \langle V \rangle$ is determined (up to isomorphism) by the choice of a graded vector space $V$ and a graded linear map $d:V[1]\to \cA$. Such a map prescribes the differentials of the formal elements $v \in \cA \langle V \rangle$.

\begin{comment}
\begin{remark}
    If $\cA \subseteq \cA\langle V\rangle$ is a Hirsch extension, the CDGA $\cA \langle V \rangle$ has new generators $x_v$ that correspond bijectively to the basis elements $v\in V$ with degree 
    $|x_v|=|v|-1$. We often abuse notation and still write $v$ for these formal symbols $x_v$. When the extension is of degree $k$, the formal symbols $v\in \cA \langle V \rangle$ all have degree $k$.
\end{remark}
\end{comment}

\begin{comment}
\begin{definition}\label{defsullivanminimal}
    A \textit{relative Sullivan algebra} is a CDGA homomorphism that is a transfinite composition of Hirsch extensions. In more detail, a CDGA map $f:\cA \to \cB$ is a relative Sullivan algebra if there exists a (possibly transfinite) exhaustive filtration $$
    \cA =\cB_0 \subseteq \cB_1 \subseteq \cdots \subseteq \cB_\beta \subseteq \cB_{\beta+1} \subseteq \cdots \subseteq \bigcup_{\beta < \alpha} \cB_\beta =\cB
    $$
    such that \begin{enumerate}
        \item each inclusion $\cB_\beta \subseteq \cB_{\beta+1}$ is a Hirsch extension,
        \item the filtration is continuous at every limit ordinal $\beta < \alpha$,
        \item the map $f$ is the transfinite composite of the sequence.
    \end{enumerate}
    In other words, a relative Sullivan algebras is a relative $\Lambda I$-cell complex, where $\Lambda I$ is the set of inclusions $\Lambda S^k \subseteq \Lambda D^k$ for $k\geq 0$. The class of relative Sullivan algebras is the smallest weakly saturated class in CDGA that contains all Hirsch extensions.
\end{definition}
\end{comment}
\subsubsection{Minimal algebras} A \textit{Sullivan algebra} an $\Lambda I$-cell complex, or equivalently, a CDGA that is realized as a transfinite colimit of Hirsch extensions over the base field $\Q$. A \textit{minimal Sullivan algebra} (or simply \textit{minimal algebra}) is a CDGA $\cM$ that admits an exhaustive \textit{skeletal filtration} $$
    \Q=\cM_0 \subseteq \cM_1 \subseteq \cM_2 \subseteq \cM_3 \subseteq \cdots \subseteq \bigcup_{k\geq 0} \cM_k=\cM
    $$
    for which each $\cM_{k-1} \subseteq \cM_{k}$ is a Hirsch extension \textit{of degree} $k$. 
    When the filtration above terminates at level $k$, (i.e., when $\cM_k=\cM_i$ for all $i>k$) we say that $\cM$ is a $k$\textit{-minimal algebra}. A $k$-minimal algebra is thus generated by its elements of degree at most $ k$. If $\cA$ is a CDGA, a \textit{minimal model} of $\cA$ is a quasi-isomorphism $\cM \weak \cA$, where $\cM$ is a minimal algebra. If $k\geq 0$ is an integer, a $k$-minimal model of $\cA$ is a $(k+1)$-connected map $\cM \to \cA$ from a $k$-minimal algebra.

\begin{comment}
\begin{remark}
    Let $\cM$ be a minimal algebra. For each $k\geq 0$, denote by $\cM_k \subseteq \cM$ the subalgebra generated by all the elements of degree $\leq k$. Then each $\cM_k$ is a $k$-minimal algebra. We refer to the filtration $\{\cM_k\}$ as the \textit{skeletal filtration} of $\cM$. If $m:\cM\weak \cA$ is a minimal model, the restrictions of $m$ along its skeletal filtration $\{\cM_k\}$ define a telescope 
    \[\begin{tikzcd}
        \Q \ar[r, hook]\ar[d, "m_0"'] & \cM_2 \ar[r, hook]\ar[d, "m_1"'] & \cM_3 \ar[r, hook]\ar[d, "m_2"'] & \cM_4 \ar[r, hook]\ar[d, "m_3"'] & \cdots  & \cM \ar[d, "m"']\\
        \cA \ar[r, equals] & \cA \ar[r, equals] & \cA \ar[r, equals] & \cA \ar[r, equals] & \cdots & \cA
    \end{tikzcd}\]
in which each $m_k$ is a $k$-minimal model of $\cA$.
\end{remark}

\begin{theorem}\label{pointwise_uniqueness}
    Let $0\leq k \leq \infty$ and $\cM \to \cA$ and $\cM' \to \cA$ be two $k$-minimal models for a simply connected CDGA $\cA$. Then there exists an isomorphism $\cM\cong \cM'$ that  is unique up to homotopy.
\end{theorem}
\begin{proof}
    The case $k=\infty$ is \cite[Theorem~14.12]{felixrathom} and the case $k<\infty$ follows readily from the fact that any map $f:\cM\to \cM'$ must preserve the skeletal filtrations $\{\cM_k\}, \{\cM'_k\}$ of $\cM$ and $\cM'$, and that the restrictions $f_k:\cM_k \to \cM_k'$ are all isomorphisms when $f$ is an isomorphism.
\end{proof}

\end{comment}
\subsubsection{Approximation telescope} Sullivan's original conception \cite{sullivan} of minimal models was to construct the smallest possible CDGA approximating the cohomology of CDGA $\mathcal{A}$ degree-by-degree. The $k$-th approximation error of a $(k-1)$-minimal model $m : \mathcal{M} \to \cA$ is measured by $k$-th cohomology $H^k(C_m)$ of the (cochain) mapping cone. The defect is removed at the CDGA level by performing Hirsch extensions
   % https://q.uiver.app/#q=WzAsNSxbMCwwLCJcXExhbWJkYSBTXntrKzF9KEheayBDX20pIl0sWzAsMiwiXFxMYW1iZGEgRF57aysxfShIXmtDX20pIl0sWzIsMCwiXFxtYXRoY2Fse019Il0sWzIsMiwiXFxvdmVybGluZXtcXG1hdGhjYWx7TX19Il0sWzQsMCwiXFxtYXRoY2Fse0F9Il0sWzAsMSwiIiwwLHsic3R5bGUiOnsidGFpbCI6eyJuYW1lIjoibW9ubyJ9fX1dLFswLDJdLFsxLDNdLFsyLDMsIiIsMix7InN0eWxlIjp7InRhaWwiOnsibmFtZSI6Im1vbm8ifX19XSxbMiw0LCJtIl0sWzMsNCwiXFxvdmVybGluZXttfSIsMl0sWzMsMCwiIiwyLHsic3R5bGUiOnsibmFtZSI6ImNvcm5lciJ9fV0sWzIsNF0sWzMsNF1d
\begin{equation} \begin{tikzcd}[sep = small] \label{pointwise_induction_lemma}
	{\Lambda S^{k+1}(H^k C_m)} && {\mathcal{M}} && {\mathcal{A}} \\
	\\
	{\Lambda D^{k+1}(H^kC_m)} && {\overline{\mathcal{M}}}
	\arrow[from=1-1, to=1-3]
	\arrow[tail, from=1-1, to=3-1]
	\arrow["m", from=1-3, to=1-5]
	\arrow[from=1-3, to=1-5]
	\arrow[tail, from=1-3, to=3-3]
	\arrow[from=3-1, to=3-3]
	\arrow["\lrcorner"{anchor=center, pos=0.125, rotate=180}, draw=none, from=3-3, to=1-1]
	\arrow["{\overline{m}}"', from=3-3, to=1-5]
	\arrow[from=3-3, to=1-5]
\end{tikzcd} \end{equation} to produce a $k$-minimal model $\overline{m} : \overline{\cM} \to \cA$. The attaching maps and extension $\overline{m}$ are determined by cocycle representatives in the cone $(v, a) \in Z^kC_m$ where $d \alpha = v$ and $\overline{m}(\alpha) = a$ for the class $\alpha = [(v,a)] \in H^k C_m$. When $\mathcal{A}$ is simply-connected, the colimit of this process produces a minimal model of $\mathcal{A}$ \cite[10.3]{griffithsmorgan} that is unique up isomorphism.

\subsubsection{Homotopy functorialty} Minimality comes at the expense of functoriality, which complicates matters in the persistent setting. Namely, for any map of simply-connected CDGAs $f: \cA \to \cB$ with minimal models $m : \cM \to \cA$ and $n: \cN \to \cB$, there exists a map $\tilde{f} :  \cM \to \cN$ fitting into a homotopy commutative square of CDGAs
\begin{equation} \label{homotopy-functoriality-sullivan-rep}
\begin{tikzcd}[row sep = 3em, column sep = 3em]
    \cM \ar[r, "\tilde{f}"]  \ar[d, "m"'] & \cN \ar[dl, Rightarrow, "H"] \ar[d, "n"]  \\ 
    \cA \ar[r, "f"'] & \cB
\end{tikzcd} 
\end{equation}
where any two choices of $\tilde{f}$ are homotopic \cite[11.5]{griffithsmorgan}. The map $\tilde{f}$ is called a \textit{Sullivan representative} of $f$. Note that it not always possible to make such a square strictly commutative.

\begin{comment}
\begin{notation}
    A recurring theme is that cohomology classes must be attached back into CDGAs as cochain representatives. To this end, roman lowercase letters $a,b,c \in \mathcal{A}$ are used for elements in a CDGA or cochain complex; elements in graded vector spaces (often the cohomology of a complex) are represented by greek letters $\alpha,\beta,\gamma \in V$ where $V$ is often $H \cA$. 
\end{notation}

\begin{theorem} \label{classical-cone-construction}
    Any simply-connected CDGA $\cA$ admits a filtered minimal model via an inductive cone construction.
\end{theorem}
\begin{proof}
    We reproduce the proof of this classical result to emphasize the analogy with \Cref{minimal_models_maps}. Under the assumption that $\cA$ is simply-connected, the unit map $\Q \to \cA$ is a 1-minimal model. Iteratively applying \Cref{pointwise_induction_lemma} produces a sequence of extensions $$
    \begin{tikzcd}
        \Q \ar[r, hook]\ar[d] & \cM_2 \ar[r, hook]\ar[d] & \cM_3 \ar[r, hook]\ar[d] & \cM_4 \ar[r, hook]\ar[d] & \cdots \\
        \cA \ar[r, equals] & \cA \ar[r, equals] & \cA \ar[r, equals] & \cA \ar[r, equals] & \cdots
    \end{tikzcd}
    $$
    in which $\cM_k \to \cA$ is a $k$-minimal model for every $k\geq 1$. It is now straightforward to check that the induced map $$
    \cM=\bigcup_{k\geq 1} \cM_k \to \cA
    $$
    is a minimal model.
\end{proof}
\end{comment}
 
\subsection{Rational homotopy theory} CDGAs and minimal models are the basic algebraic tools to perform calculations in rational homotopy theory. We briefly review the literature, focusing on the results we will generalize to the persistent setting.  
\subsubsection{The key adjunction} Let $\Delta$ be the simplex category. There is a simplicial object $\Omega_\bullet : \Delta^{op} \to \CDGA_\Q$ that associates to each $[n]$ the CDGA 
$$\Omega_n = \Lambda(x_1,dx_1, \cdots, x_n, dx_n)$$ of polynomial differential forms with rational coefficients over the $n$-simplex with degrees $\lvert x_i \rvert = 0$ for all $i$. Kan extension of this simplicial object over the Yoneda embedding yields an adjunction
\[\begin{tikzcd}
            \langle - \rangle : \CDGA_\Q \arrow[r, shift left=1ex, ""{name=G}] & \sSet^\op : \apl \arrow[l, shift left=.5ex, ""{name=F}]
            \arrow[phantom, from=F, to=G, , "\scriptscriptstyle\boldsymbol{\bot}"]
\end{tikzcd}.\] The functor $\mathcal{A}_{PL}$ maps a simplicial set $X$ to that CDGA of polynomial differential forms. The \textit{spatial realization} functor $\langle - \rangle$ sends a CDGA $\mathcal{A}$ to the simplicial set of cosingular simplices $\langle \mathcal{A} \rangle = \text{Map}(\mathcal{A}, \Omega_\bullet)$.
\subsubsection{Postnikov towers} Rational homotopy groups of simply connected spaces can be computed via minimal CDGA models of their rational PL differential forms. For a simply-connected simplicial set $X$ with minimal model $m : \mathcal{M} \xrightarrow{\sim} \apl(X)$, the spatial realization functor
% https://q.uiver.app/#q=WzAsOSxbMCwwLCJcXGJpZ290aW1lcyBcXExhbWJkYSBTXmsiXSxbMCwyLCJcXGJpZ290aW1lcyBcXExhbWJkYSBEXmsiXSxbMiwwLCJcXG1hdGhjYWx7TX1fe2stMX0iXSxbMiwyLCJcXG1hdGhjYWx7TX1fayJdLFszLDEsIlxceHJpZ2h0YXJyb3d7XFxsYW5nbGUgLSBcXHJhbmdsZX0iXSxbNCwwLCJcXGxhbmdsZSBcXG1hdGhjYWx7TX1fa1xccmFuZ2xlIl0sWzYsMiwiSyhcXHBpX2soWCkgXFxvdGltZXMgXFxRLCBrLTEpIl0sWzYsMCwiUEsoXFxwaV9rKFgpIFxcb3RpbWVzIFxcUSwgay0xKSJdLFs0LDIsIlxcbGFuZ2xlIFxcbWF0aGNhbHtNfV97ay0xfVxccmFuZ2xlIl0sWzAsMSwiIiwyLHsic3R5bGUiOnsidGFpbCI6eyJuYW1lIjoibW9ubyJ9fX1dLFswLDIsIlxcb3RpbWVzIFxcdGF1Il0sWzEsM10sWzIsMywiIiwwLHsic3R5bGUiOnsidGFpbCI6eyJuYW1lIjoibW9ubyJ9fX1dLFszLDAsIiIsMSx7InN0eWxlIjp7Im5hbWUiOiJjb3JuZXIifX1dLFs3LDYsIiIsMix7InN0eWxlIjp7ImhlYWQiOnsibmFtZSI6ImVwaSJ9fX1dLFs1LDgsIiIsMCx7InN0eWxlIjp7ImhlYWQiOnsibmFtZSI6ImVwaSJ9fX1dLFs4LDYsIlxccHJvZCBcXGxhbmdsZSBcXHRhdSBcXHJhbmdsZSJdLFs1LDddLFs1LDYsIiIsMSx7InN0eWxlIjp7Im5hbWUiOiJjb3JuZXIifX1dXQ==
\begin{equation} \label{classical-postnikov} \begin{tikzcd}[sep = tiny]
	{\bigotimes \Lambda S^{k+1}} && {\mathcal{M}_{k-1}} && {\langle \mathcal{M}_k\rangle} && {PK(\pi_k^\Q(X), k+1)} \\
	&&& {\xrightarrow{\langle - \rangle}} \\
	{\bigotimes \Lambda D^{k+1}} && {\mathcal{M}_k} && {\langle \mathcal{M}_{k-1}\rangle} && {K(\pi_k^\Q(X), k+1)}
	\arrow["{\otimes \tau}", from=1-1, to=1-3]
	\arrow[tail, from=1-1, to=3-1]
	\arrow[tail, from=1-3, to=3-3]
	\arrow[from=1-5, to=1-7]
	\arrow[two heads, from=1-5, to=3-5]
	\arrow["\lrcorner"{anchor=center, pos=0.125}, draw=none, from=1-5, to=3-7]
	\arrow[two heads, from=1-7, to=3-7]
	\arrow[from=3-1, to=3-3]
	\arrow["\lrcorner"{anchor=center, pos=0.125, rotate=180}, draw=none, from=3-3, to=1-1]
	\arrow["{\prod \langle \tau \rangle}", from=3-5, to=3-7]
\end{tikzcd} \end{equation}
transforms cell attachments in the minimal model into rational principal fibrations in the rational Postnikov tower \cite[Thm. 3.3]{Deligne1975}. This implies we have a correspondence
$$\Big\{ \parbox{7em}{\centering $k$-cells of $\mathcal{M}$} \Big\} \hspace{2em} \longleftrightarrow \hspace{2em} \Big\{ \parbox{10em}{ \centering generators of $\pi_k^\Q(X)$} \Big\} \vspace{1em}.$$ where the attaching maps in the minimal model correspond to components of the rational classfying map, completely determining the rational homotopy type of $X$.

\subsection{Persistence theory.} For a fixed integer $[n]$ let $[n]$ be the category associated to the poset $\{0,1,2, \ldots, n\}$. The category $p\cC$ of \textit{persistent objects} in a category $\cC$ is the functor category $\Fun([n], \cC)$. If letters $X,Y,Z, \cdots$ are used for the objects of the category $\cC$, we use corresponding boldface symbols $\X,\Y,\Z, \cdots $ to name the objects of $p\cC$. Dually, we label the functor category $\Fun([n]^{op}, \cC)$ of \textit{copersistent objects }as $p^*\cC$, where $[n]^{op}$ is the dual poset to $[n]$.

Let $\mathsf{Vec}_\mathbf{k}$ be the category of vector spaces over a field $\mathbf{k}$. A \textit{persistence module} is a functor $\mathbb{V} : [n] \to \mathsf{Vec}_\mathbf{k}$, or equivalently, a persistent object in $\mathsf{Vec}_\mathbf{k}$. Since we are primarily concerned with rational homotopy theory, we will work over $\mathbf{k} = \Q$ from henceforth.

\begin{definition}
    Let $0\leq s<t\leq n+1$. The \textit{interval module} $\bI_{[s,t)} : [n] \to \mathsf{Vec}_\Q$ is the persistence module defined by
    $$
    \bI_{[s,t)}(i) = \begin{cases} \Q & \text{ if } s\leq i < t \\ 0 & \text{ else} \end{cases}
    $$
    and, at the level of maps, 
    $$\bI_{[s,t)}(i\leq j) = \begin{cases} \text{id}_\Q & s\leq i \leq j < t\\ 0 \text{otherwise} \end{cases}$$
\end{definition}

Direct sums of persistence modules are constructed degree-wise on both objects and morphisms. 
Interval modules are the atomic units into which all persistence modules decompose, as stated below.

\begin{theorem}\cite[Thm. 1]{Webb_1985} \label{interval-decomposition}
    If $\V : [n] \to \mathsf{Vec}_\Q$ is a persistence module then there exists a set of intervals $[s_\gamma, t_\gamma) \subseteq [0,\infty)$ such that $\V$ splits as a direct sum of interval modules
    $$\V \cong \bigoplus_{\gamma \in \Gamma} \bI_{[s_\gamma,t_\gamma)}.$$
    Moreover, this decomposition is unique up to permutation of summands.
\end{theorem}

\begin{remark}
    There are a number of analogous theorems in the TDA literature \cite{Carlsson05,chazal2016structure, Crawley-Boevey_2015}. We refer to this one since it most explicitly emphasizes point-wise vector spaces of arbitrary dimension.
\end{remark}

\section{Minimal models of maps}

The classical procedure for constructing Sullivan representatives of maps is to invoke an abstract lifting property. The point of this section is to produce explicit formulae for the Sullivan representative. This entails the inductive construction of an approximation telescope over a map, without resorting to lifting properties.

\subsection{Basic definitions for maps} \label{MMmapssection} We begin by laying out the basic definitions required for minimal models of maps. Let $f:\cA \to \cB$ be a map of simply-connected CDGAs. A \textit{$k$-minimal model} for $f$ is a diagram
% https://q.uiver.app/#q=WzAsNSxbMCwwLCJcXG1hdGhjYWx7TX0iXSxbMiwwLCJcXG1hdGhjYWx7Tn0iXSxbMCwyLCJcXG1hdGhjYWx7QX0iXSxbMiwyLCJcXG1hdGhjYWx7Qn0iXSxbMywxLCJcXGluIFxcQ0RHQV9cXFEiXSxbMCwyLCIgbSIsMl0sWzIsMywiZiIsMl0sWzAsMSwiZiciXSxbMSwzLCJuIl0sWzEsMiwiSCIsMCx7InNob3J0ZW4iOnsic291cmNlIjoxMCwidGFyZ2V0IjoxMH0sImxldmVsIjoyfV1d
\begin{equation} \begin{tikzcd}[row sep=tiny, column sep = small] \label{k-sullivanrep-square} 
	{\mathcal{M}} && {\mathcal{N}} \\
	&&& {\text{in } \CDGA_\Q} \\
	{\mathcal{A}} && {\mathcal{B}}
	\arrow["{f'}", from=1-1, to=1-3]
	\arrow["{ m}"', from=1-1, to=3-1] 
	\arrow["H", shorten <=6pt, shorten >=6pt, Rightarrow, from=1-3, to=3-1]
	\arrow["n", from=1-3, to=3-3]
	\arrow["f"', from=3-1, to=3-3]
\end{tikzcd} \end{equation} that commutes up to a specified homotopy $H$, and for which $m$ and $n$ are $k$-minimal models of $\cA$ and $\cB$ respectively. Integration  yields a homotopy commutative square and induced cone map 
% https://q.uiver.app/#q=WzAsNCxbMCwwLCJcXG1hdGhjYWx7TX0iXSxbMiwwLCJcXG1hdGhjYWx7Tn0iXSxbMCwyLCJcXG1hdGhjYWx7QX0iXSxbMiwyLCJcXG1hdGhjYWx7Qn0iXSxbMCwyLCIgbSIsMl0sWzIsMywiZiIsMl0sWzAsMSwiZiciXSxbMSwzLCJuIl0sWzEsMiwiXFxpbnRfMF4xIEgiLDAseyJzaG9ydGVuIjp7InNvdXJjZSI6MTAsInRhcmdldCI6MTB9LCJsZXZlbCI6Mn1dXQ==
\[\begin{tikzcd}[column sep = scriptsize]
	{\mathcal{M}} && {\mathcal{N}} \\
	\\
	{\mathcal{A}} && {\mathcal{B}}
	\arrow["{f'}", from=1-1, to=1-3]
	\arrow["{ m}"', from=1-1, to=3-1]
	\arrow["{\int_0^1 H}", shorten <=6pt, shorten >=6pt, Rightarrow, from=1-3, to=3-1]
	\arrow["n", from=1-3, to=3-3]
	\arrow["f"', from=3-1, to=3-3]
\end{tikzcd} \hspace{4em}
\varphi_{H} = \begin{bmatrix} f' & 0 \\ \int_0^1 H & f \end{bmatrix} : C_{m} \to C_{n}.\hspace{4em} \text{in } \mathsf{Ch}_\Q^*.\] The mapping cones $C_m, C_n$ are the homotopy cofibers of $m,n$ in the category $\Ch_\Q$, carry no natural algebra structure, and have trivial cohomology up-to degree $k$.

\begin{comment}
Given such $k$-minimal model of $f$, there is an induced cochain map on the cones, given in matrix notation by
$$
\varphi_{H} = \begin{bmatrix} f' & 0 \\ \int_0^1 H & f \end{bmatrix} : C_{m} \to C_{n}.
$$ 

\begin{remark} Note that the mapping cones $C_m, C_n$ are the homotopy cofibers of $m,n$ in the category $\Ch_\Q$, and thus carry no natural algebra structure. Here $\int_0^1 H$ is the cochain homotopy induced by the CDGA homotopy $H$. \end{remark}
\end{comment}
\subsubsection{Hirsch extensions} \label{hirschremark} A \textit{Hirsch extension} of a CDGA map $f : \mathcal{A} \to \mathcal{B}$ is a commutative square 
$$
        \begin{tikzcd}
            \cA \ar[d, hook]\ar[r, "f"] & \cB \ar[d, hook] \\
            \cA\langle V \rangle \ar[r, "\overline{f}"'] & \cB\langle W \rangle.
        \end{tikzcd}
        $$ in which $\cA \subseteq \cA \langle V \rangle$ and $\cB \subseteq \cB \langle W \rangle$ are Hirsch extensions of $\cA$ and $\cB$ along graded vector spaces $V$ and $W$. A Hirsch extension of $f$ is said to be \textit{of degree} $k$ if both the extensions $\cA \subseteq \cA \langle V \rangle$ and $\cB \subseteq \cB \langle W \rangle$ are of degree $k$. A Hirsch extension of $f$ is uniquely determined by the data of graded vector spaces $V,W$ and the choice of graded linear maps 
    $$
    d:V[1]\to \cA
    \:\:\:\:\:\:\:\:\:\:\:\:\:\:\:\:\:\:\:\:\:\:\:\:\:\:\:\:\:\:\:\:\ 
    d:W[1]\to \cB \:\:\:\:\:\:\:\:\:\:\:\:\:\:\:\:\:\:\:\:\:\:\:\:\:\:\:\:\:\:\:\:\ 
    \overline{f}:V \to \cB\langle W \rangle
    $$
    satisfying the compatibility condition $d\overline{f}(v)=f(dv)$ for each $v\in V$. In other words, the following diagram of cochain complexes should commute 
    $$
    \begin{tikzcd}
        {V[1]} \ar[r, "{\overline{f}[1]}"]\ar[d, "d"'] & {\cB\langle W[1] \rangle } \ar[d, "d"] \\
        {\cA} \ar[r, "{f}"] & {\cB}.
    \end{tikzcd}
    $$
\begin{comment}
\subsubsection{Cones} \label{MMmapssection} Let $f:\cA \to \cB$ be a map of simply-connected CDGAs. A $k$-minimal model for $f$ is a diagram
\[\begin{tikzcd}
    \label{k-sullivanrep-square} 
            \cM \ar[r, "f'"] \ar[d, "m"']   & \cN \ar[dl, Rightarrow, "H", shorten = 1.0ex] \ar[d, "n"]\\
            \cA \ar[r, "f"'] & \cB 
    \end{tikzcd}\]
which commutes up to a specified homotopy $H$, and for which $m,n$ are $k$-minimal models of $\cA, \cB$ respectively. Given such $k$-minimal model of $f$, there is an induced cochain map on the cones, given in matrix notation by :
$$ 
\varphi_{H} = \begin{bmatrix} f' & 0 \\ \int_0^1 H & f \end{bmatrix} : C_{m} \to C_{n}.
$$ 

\begin{remark} Note that the mapping cones $C_m, C_n$ are the homotopy cofibers of $m,n$ in the category $\Ch_\Q$, and thus carry no natural algebra structure. Here $\int_0^1 H$ is the cochain homotopy induced by the CDGA homotopy $H$. \end{remark}
\end{comment}
\subsection{Homotopy-commutative telescopes} We continue with a technical lemma concerning homotopy-commutative telescopes and their colimits. 

\begin{lemma}\label{towerlemma}
    Suppose given a tower of homotopy commutative squares of CDGAs
\[\begin{tikzcd}[row sep = tiny, column sep = tiny]
	{\cM_0} && {\cM_1} && {\cM_2} && \cdots \\
	& {\cN_0} & {{}} & {\cN_1} & {{}} & {\cN_2} && \cdots \\
	{\cA_0} & {{}} & {\cA_1} & {{}} & {\cA_2} & {{}} & \cdots \\
	& {\cB_0} && {\cB_1} && {\cB_2} && \cdots
	\arrow[from=1-1, to=2-2]
	\arrow[from=1-1, to=3-1]
	\arrow[from=2-2, to=4-2]
	\arrow[from=3-1, to=4-2]
	\arrow["{H_0}", Rightarrow, from=2-2, to=3-1]
	\arrow[from=1-1, to=1-3]
	\arrow[from=2-2, to=2-4]
	\arrow[from=4-2, to=4-4]
	\arrow[from=1-3, to=2-4]
	\arrow[from=2-4, to=4-4]
	\arrow[from=3-3, to=4-4]
	\arrow["{H_1}", Rightarrow, from=2-4, to=3-3]
	\arrow[from=1-3, to=1-5]
	\arrow[from=2-4, to=2-6]
	\arrow[from=4-4, to=4-6]
	\arrow[""{name=0, anchor=center, inner sep=0}, from=2-6, to=4-6]
	\arrow[from=1-5, to=2-6]
	\arrow[from=3-5, to=4-6]
	\arrow["{H_2}", Rightarrow, from=2-6, to=3-5]
	\arrow[from=1-5, to=1-7]
	\arrow[from=2-6, to=2-8]
	\arrow[from=4-6, to=4-8]
	\arrow[no head, from=1-3, to=2-3]
	\arrow[from=2-3, to=3-3]
	\arrow[no head, from=3-1, to=3-2]
	\arrow[from=3-2, to=3-3]
	\arrow[no head, from=3-3, to=3-4]
	\arrow[from=3-4, to=3-5]
	\arrow[from=3-6, to=3-7]
	\arrow[from=2-5, to=3-5]
	\arrow[no head, from=1-5, to=2-5]
	\arrow[shorten >=5pt, no head, from=3-5, to=0]
\end{tikzcd}\] 
    each of which commutes up to a specified homotopy $H_i$. Suppose moreover that each homotopy $H_{i+1}$ is an extension of $H_i$, in the sense that the following squares commute for each $i\geq 0$
    $$
    \begin{tikzcd}
        \cM_i \ar[r] \ar[d, "H_i"'] & \cM_{i+1} \ar[d, "H_{i+1}"] \\
        \cB_i \otimes \Lambda(t,dt) \ar[r] & \cB_{i+1} \otimes \Lambda(t,dt).
    \end{tikzcd}
    $$
    In this situation, the induced map $H=\colim_i\:H_i$ renders homotopy commutative the square of colimits 
    $$
    \begin{tikzcd}
            \colim_i\:\cM_i \ar[rr] \ar[d]   && \colim_i\:\cN_i \ar[dll, Rightarrow, "H", shorten = 1.0ex] \ar[d]\\
            \colim_i\:\cA_i \ar[rr] && \colim_i\:\cB_i .
    \end{tikzcd}
    $$
\end{lemma}
\begin{proof}
    Denote $\cA=\colim_i \:\cA_i$, and similarly for $\cB, \cM, \cN$. Since the tensor product of CDGAs commutes with colimits in both variables, $H$ indeed provides an algebra homomorphism $$
    \cM \longrightarrow \cB \otimes \Lambda(t,dt),
    $$ in other words a CDGA homotopy between maps $\cM\to \cB$. It therefore suffices to check that the two composites of $H$ with the augmentations $\varepsilon_0, \varepsilon_1:\Lambda(t,dt) \to \Q$ give the two appropriate composites on the square of colimits. This follows immediately from the fact that the diagram $$
    \begin{tikzcd}
        \cM_i \ar[r, "H_i"] \ar[d] & \cB_i\otimes \Lambda(t,dt) \ar[d] \ar[dr, "{\varepsilon_0, \varepsilon_1}"] & \\
        \cM \ar[r, "H"'] & \cB \otimes \Lambda(t,dt) \ar[r, "{\varepsilon_0, \varepsilon_1}"'] & \cB
    \end{tikzcd}
    $$
    commutes for each $i\geq 0$, where the vertical maps are induced by the canonical inclusions $\cM_i \to \cM$ and $\cB_i \to \cB$ into the colimits.
\end{proof}

\subsection{Approximation telescopes of maps} We now define a procedure for extending the $(k-1)$-minimal model of a map to a $k$-minimal model. Analogous to the classical case, we do this via attaching the cohomology of the mapping cone - albeit with the additional information of the structure map between the cohomology of the source and target cones to keep track of. 

\begin{comment}
\begin{theorem} \label{squares_theorem}
    Suppose we have $(k-1)$-Sullivan representative $$\tilde{f}(k-1) : \cM_\cA(k-1) \to \cM_\cB(k-1)$$ i.e. a diagram
    \begin{center}
        \begin{tikzcd}[row sep = 3em]
            \mathcal{M}_\cA(k-1) \ar[r, "\tilde{f}(k-1)"] \ar[d, "m_\cA(k-1)"']   & \mathcal{M}_\cB(k-1) \ar[dl, Rightarrow, "H(k-1)", shorten = 2.0ex] \ar[d, "m_\cB(k-1)"]\\
            A \ar[r, "f"']  & B.
        \end{tikzcd}
    \end{center}
    and let $\varphi_{H(k-1)} =  C_{m_\cA(k-1)} \to C_{m_\cB(k-1)}$ be the induced map on the respective cones. Then there exists a $k$-Sullivan representative
     \begin{center}
        \begin{tikzcd}[row sep = 3em]
            \mathcal{M}_\cA(k) \ar[r, "\tilde{f}(k)"] \ar[d, "m_\cA(k)"']   & \mathcal{M}_\cB(k) \ar[dl, Rightarrow, "H(k)", shorten = 2.0ex] \ar[d, "m_\cB(k-1)"]\\
            A \ar[r, "f"'] & B
        \end{tikzcd}
    \, \, such that  \, \, 
    \begin{tikzcd}[row sep = 3em, column sep = 4em]
        Q^k (\cM_\cA(k)) \ar[r, "Q^k(\tilde{f}(k))"] \ar[d, "\cong"'] & Q^k(\cM_\cA(k)) \ar[d, "\cong"] \\
        H^{k+1} C_{m_\cA(k-1)} \ar[r, "H^{k+1} \varphi_{H(k-1)}"] & H^{k+1} C_{m_\cB(k-1)}
    \end{tikzcd}
    \end{center}
\end{theorem} 
Essentially, given a $(k-1)$-Sullivan representative and choice of homotopy $H(k-1)$, one can compute linear part of $\tilde{f}(k)$ in degree $k$ via the cohomology of a map between the cones. 
\end{comment} 

\begin{lemma} \label{inductivestep}
    Let $f:\cA \to \cB$ be a map of CDGAs and $k\geq 2$ an integer. Suppose given a $(k-1)$-minimal model $g$ of $f$
    $$
    \begin{tikzcd}[column sep = scriptsize]
        \cM \ar[rr, "g"] \ar[d, "m"']   && \cN \ar[dll, Rightarrow, "H", shorten = 2.0ex] \ar[d, "n"]\\
        \cA \ar[rr, "f"']  && \cB.
    \end{tikzcd}
    $$
    and denote by $\varphi:C_m \to C_n$ the induced cochain map on the cones. There exists an extension to a $k$-minimal model 
    $$
    \begin{tikzcd}[column sep = small]
        \cM\langle H^kC_m \rangle \ar[rr, "\overline{g}"] \ar[d, "\overline{m}"']   && \cN\langle H^kC_n \rangle \ar[dll, Rightarrow, "\overline{H}", shorten = 2.0ex] \ar[d, "\overline{n}"]\\
        \cA \ar[rr, "f"']  && \cB.
    \end{tikzcd}
    $$
    for which the map induced by $\overline{g}$ on indecomposables of degree $k$ is $Q^k(\overline{g})=H^k\varphi$.
\end{lemma}

\begin{proof}
    Let us denote by $\psi:V\to W$ the linear map 
    $
    H^k\varphi : H^kC_m \to H^kC_n
    $
    induced by $\varphi$ on degree $k$ cohomology. We wish to define a Hirsch extension of $g$ of degree $k$
    $$
    \begin{tikzcd}
            \cM \ar[d, hook]\ar[r, "{g}"] & \cN \ar[d, hook] \\
            \cM\langle V \rangle \ar[r, "{\overline{g}}"'] & \cN\langle W \rangle.
    \end{tikzcd}
    $$
    with the property that $Q^k(\overline{g})=\psi$. To do so, we need only to specify graded linear maps
    $$
    V[1] \to \cM \:\:\:\:\:\:\:\:\:\:\:\:\:\:\:\:\:\:\:\:\:\:\:\:\:\:\:\:\:\:\:\:\ 
    W[1] \to \cN \:\:\:\:\:\:\:\:\:\:\:\:\:\:\:\:\:\:\:\:\:\:\:\:\:\:\:\:\:\:\:\:\:
    V \to \cN\langle W \rangle
    $$
    satisfying a suitable compatibility condition, as stated in \Cref{hirschremark}.
    
    We can choose bases $\{\epsilon_i,\alpha_j\}$ of $V$ and $\{\psi(\epsilon_i),\beta_\ell\}$ of $W$ that are adapted to the decomposition $$
    \Coim(\psi)\oplus \Ker(\psi) \xrightarrow{\psi} \im(\psi)\oplus \Coker(\psi),
    $$
    so that $\{\epsilon_i\}$ is a basis of $\Coim(\psi) \cong \im(\psi)$ and $\{\alpha_j\}, \{\beta_\ell\}$ are bases for $\Ker(\psi)$ and $\Coker(\psi)$ respectively. We now pick representatives for the cohomology classes $\epsilon_i, \alpha_j \in H^kC_m$ and $\beta_\ell\in H^kC_n$, which we denote $(v_i,a_i), (v'_j, a'_j) \in Z^kC_m$ and $(w_\ell, b_\ell)\in Z^kC_n$ respectively. Since they are cocycles in the respective cones, the following relations hold:
    $$
    \begin{aligned}
        dv_i&=0 \\
        da_i&=m(v_i)
    \end{aligned}
    \:\:\:\:\:\:\:\:\:\:\:\:\:\:\:\:\:\:\:\:\:\:\:\:\:\:\:\:\:\:\:\:\:
    \begin{aligned}
        dv'_j&=0 \\
        da'_j&=m(v'_j)
    \end{aligned}
    \:\:\:\:\:\:\:\:\:\:\:\:\:\:\:\:\:\:\:\:\:\:\:\:\:\:\:\:\:\:\:\:\:
    \begin{aligned}
        dw_\ell&=0 \\
        db_\ell&=n(w_\ell).
    \end{aligned}
    $$
    Moreover, since $\psi(\alpha_j)=0$, there are elements $(x_j, y_j)\in C_n$ with $$
    d(x_j, y_j)=\varphi(v_j', a_j')=\big( \: g(v'_j)\:,\: f(a'_j) + \int_0^1 H(v'_j) \: \big).
    $$ 
    Let the following assignments define the required linear graded maps: $$
    \begin{aligned}
        d:V[1] &\to \cM \\
        \epsilon_i &\mapsto v_i  \\  
        \alpha_j &\mapsto v'_j 
    \end{aligned}
    \:\:\:\:\:\:\:\:\:\:\:\:\:\:\:\:\:\:\:\:\:\:\:\:\:\:\:\:\:\:\:\:\:
    \begin{aligned}
        d:W[1] &\to \cN \\
        \psi(\epsilon_i) &\mapsto g(v_i) \\
        \beta_\ell &\mapsto w_\ell
    \end{aligned}
    \:\:\:\:\:\:\:\:\:\:\:\:\:\:\:\:\:\:\:\:\:\:\:\:\:\:\:\:\:\:\:\:\:
    \begin{aligned}
        \overline{g}:V &\to \cN\langle W \rangle \\
        \epsilon_i &\mapsto \psi(\epsilon_i) \\
        \alpha_j &\mapsto x_j.
    \end{aligned}
    $$
    These data satisfy the required compatibility condition of \Cref{hirschremark} because $d\overline{g}(\epsilon_i)=d\psi(\epsilon_i)=g(dv_i)$ and $d\overline{g}(\alpha_j)=dx_j=g(v'_j)$. 
    
    We have thus defined a Hirsch extension $\overline{g}$ of $g$. Note that since $x_j\in \cN^k$, it must be a non-trivial product of generators because $\cN$ is a free algebra on generators of degree $<k$. This ensures that $Q^k(\overline{g})(\alpha_j)=0$ and consequently $Q^k(\overline{g})=\psi$ as desired. 
    
    At this point the extension $\cM\langle V \rangle$ contains formal symbols $\epsilon_i, \alpha_j$ in bijection with a basis for $V=H^kC_m$, with the property that $d\epsilon_i = v_i$ and $d\alpha_j=v'_j$. Similarly $\cN\langle W \rangle$ contains formal symbols $\psi(\epsilon_i), \beta_\ell$ in bijection with a basis for $W=H^kC_n$, such that $d\psi(\epsilon_i)=g(v_i)$ and $d\beta_\ell=w_\ell$.

    We now proceed to extend the maps $m:\cM \to \cA$, $n:\cN \to \cB$ along the inclusions $\cM \subseteq \cM\langle V \rangle$, $\cN \subseteq \cN \langle W \rangle$. This is done by the following assignments $$
    \begin{aligned}
        V &\to \cA \\
        \epsilon_i &\mapsto a_i \\
        \alpha_j &\mapsto a'_j
    \end{aligned}
    \:\:\:\:\:\:\:\:\:\:\:\:\:\:\:\:\:\:\:\:\:\:\:\:\:\:\:\:\:\:\:\:\:
    \begin{aligned}
        W &\to \cB \\
        \psi(\epsilon_i) &\mapsto f(a_i) + \int_0^1H(v_i) \\
        \beta_\ell &\mapsto b_\ell.
    \end{aligned}
    $$
    It needs to be checked that these assignments yield well-defined cochain maps $\overline{m}:\cM\langle V \rangle \to \cA$ and $\overline{n}:\cN\langle W \rangle \to \cB$. This follows from the easily checked facts that $d\overline{m}(\epsilon_i)=m(v_i)$, $d\overline{m}(\alpha_j)=m(v'_j)$, $d\overline{n}(\beta_\ell)=n(w_\ell)$ and the following calculation :
    $$
    \begin{aligned}
         d\overline{n}(\psi(\epsilon_i))&=f(da_i)+d\int_0^1H(v_i) \\
        &=fm(v_i)+ ng(v_i)-fm(v_i) - \int_0^1 H(dv_i) \\
        &= ng(v_i)=\overline{n}(d\psi(\epsilon_i)).
    \end{aligned}
    $$
    We claim that $\overline{m}$ and $\overline{n}$ are $k$-minimal models of $\cA$ and $\cB$. To see this, note that the respective inclusions induce isomorphisms $H^*\cM \cong H^*\cM\langle V \rangle$ and $H^*\cN \cong H^*\cN\langle W \rangle$ for $*<k$, since $V$ and $W$ are concentrated in degree $k$. The classes $\epsilon_i, \alpha_j$ that formed a basis for $H^kC_m$ are now killed in $H^kC_{\overline{m}}$, since $d(\epsilon_i,0)=(v_i, a_i)$ and $d(\alpha_j,0)=(v'_j, a'_j)$ in $C_{\overline{m}}$. Similarly the classes $\psi(\epsilon_i), \beta_\ell$ are now killed in $H^kC_{\overline{n}}$ since $d(\beta_\ell, 0)=(w_\ell, b_\ell)$ and $$
    d(\psi(\epsilon_i),0)=(g(v_i), f(a_i)+\int_0^1 H(v_i))
    $$
    in $C_{\overline{n}}$. This shows that $H^kC_{\overline{m}}=H^kC_{\overline{n}}=0$. Since $\cM\langle V \rangle$ and $\cN \langle W \rangle$ are minimal algebras by construction, it follows that $\overline{m}, \overline{n}$ are $k$-minimal models.
    
    It remains to extend the homotopy $H:\cM \to \cB \otimes \Lambda(t,dt)$ along the inclusion $\cM \subseteq \cM\langle V \rangle$. To this end, let us define $\overline{H}$ by the following assigments $$
    \begin{aligned}
        \epsilon_i &\mapsto f(a_i) + \int_0^t H(v_i)  \\
        \alpha_j &\mapsto f(a'_j) + \int_0^t H(v'_j) + d(y_j\otimes t).
    \end{aligned}
    $$
    We need to check that $\overline{H}$ is indeed a cochain map. This follows from the following computation:
    $$
    \begin{aligned}
        d\overline{H}(\epsilon_i) &=fm(v_i) + H(v_i) - fm(v_i) - \int_0^t H(dv_i) \\
        &= H(v_i)=\overline{H}(d\epsilon_i)
    \end{aligned}
    $$
    where we used that $dv_i=0$. A similar calculation yields $d\overline{H}(\alpha_j)=H(v'_j)=\overline{H}(d\alpha_j)$. It is straightforward to check that $\overline{H}$ is a homotopy from $f\overline{m}$ to $ \overline{n}\overline{g}$, which finishes the proof.
\end{proof}

\begin{remark}
    Note that when defining the extensions $\cM \subseteq \cM \langle V \rangle$ and $\cN \subseteq \cN\langle W \rangle$, we have simultaneously ensured that both $H^kC_{\overline{m}}=H^kC_{\overline{n}}=0$, and that the cocycles 
    $$
    Q([v,a])=\big(\: g(v) \:,\: f(a)+ \int_0^1 H(v) \:\big) \:\: \in Z^kC_n
    $$
    are nullcohomologous for every class $[v,a]\in V=H^kC_m$. These cocycles are exactly the obstructions to extending $g$ and $H$ to a $k$-minimal model, once the maps $\overline{m}$ and $\overline{n}$ have been specified. More details on this perspective can be found in \cite[Theorem~11.5]{griffithsmorgan}.
\end{remark}

The new maps $\overline{g}, \overline{m},\overline{n}$ and $\overline{H}$ in our construction are explicit, and extended over the finitely many generating elements we have added in $V$ and $W$. In theory, this allows one to build up $k$-minimal models of maps inductively by hand. The theorem above can be read as an algorithmic procedure that  given the input of a $(k-1)$-minimal model of $f$, outputs a $k$-minimal model. This procedure can be iterated \textit{ad infinitum} and yields a minimal model of the map $f$, generalizing the classical approximation telescope construction.

\begin{theorem}[Approximation telescope of maps] \label{minimal_models_maps}
    Any homomorphism between simply-connected CDGAs admits a minimal model built via an approximation telescope.
\end{theorem}
\begin{proof}
    Let $f:\cA \to \cB$ denote such a CDGA map. Under the assumption that $\cA$ and $\cB$ are simply-connected, the trivial square $$
    \begin{tikzcd}
        \Q \ar[r, equals] \ar[d] & \Q \ar[d] \\
        \cA \ar[r, "f"] & \cB
    \end{tikzcd}
    $$
    is a 1-minimal model of $f$. \Cref{inductivestep} ensures that it extends to a 2-minimal model
    $$
    \begin{tikzcd}
        \Lambda H^2\cA \ar[r, "\overline{f}"] \ar[d] & \Lambda H^2 \cB \ar[dl, Rightarrow, "H", shorten = 2.0ex] \ar[d] \\
        \cA \ar[r, "f"'] & \cB
    \end{tikzcd}
    $$
    where $Q^2(\overline{f})\cong H^2f$. Iterating this procedure, we obtain a tower of models of $f$ in which the square labelled $k$ is a $k$-minimal model of $f$. \Cref{towerlemma} ensures that there is a homotopy filling the colimit square.
\[\begin{tikzcd}[column sep=tiny, row sep=small]
	\Q && {\cM_2} && {\cM_3} && {\cM_4} && \cdots && {\cM_\infty} \\
	& \Q & {{}} & {\cN_2} & {{}} & {\cN_3} & {{}} & {\cN_4} && \cdots && {\cN_\infty} \\
	\cA & {{}} & \cA & {{}} & \cA & {{}} & \cA & {{}} & \cdots && \cA \\
	& \cB && \cB && \cB && \cB && \cdots && \cB
	\arrow[from=1-1, to=2-2]
	\arrow[from=1-1, to=3-1]
	\arrow[from=2-2, to=4-2]
	\arrow["f"'{pos=0.3}, from=3-1, to=4-2]
	\arrow[hook, from=1-1, to=1-3]
	\arrow[hook, from=2-2, to=2-4]
	\arrow[Rightarrow, no head, from=4-2, to=4-4]
	\arrow[from=2-4, to=4-4]
	\arrow[from=1-3, to=2-4]
	\arrow[hook, from=1-3, to=1-5]
	\arrow[hook, from=2-4, to=2-6]
	\arrow[Rightarrow, no head, from=4-4, to=4-6]
	\arrow[from=2-6, to=4-6]
	\arrow[from=1-5, to=2-6]
	\arrow[hook, from=1-5, to=1-7]
	\arrow[""{name=0, anchor=center, inner sep=0}, hook, from=2-6, to=2-8]
	\arrow[Rightarrow, no head, from=4-6, to=4-8]
	\arrow[from=1-7, to=2-8]
	\arrow[from=2-8, to=4-8]
	\arrow[hook, from=1-7, to=1-9]
	\arrow[hook, from=2-8, to=2-10]
	\arrow[Rightarrow, no head, from=4-8, to=4-10]
	\arrow[Rightarrow, no head, from=3-1, to=3-2]
	\arrow[Rightarrow, no head, from=3-2, to=3-3]
	\arrow[Rightarrow, no head, from=3-3, to=3-4]
	\arrow[Rightarrow, no head, from=3-4, to=3-5]
	\arrow[Rightarrow, no head, from=3-5, to=3-6]
	\arrow[Rightarrow, no head, from=3-6, to=3-7]
	\arrow[Rightarrow, no head, from=3-7, to=3-8]
	\arrow[Rightarrow, no head, from=3-8, to=3-9]
	\arrow[no head, from=1-3, to=2-3]
	\arrow[from=2-3, to=3-3]
	\arrow[no head, from=1-5, to=2-5]
	\arrow[from=2-5, to=3-5]
	\arrow[from=2-7, to=3-7]
	\arrow["f"'{pos=0.3}, from=3-3, to=4-4]
	\arrow["f"'{pos=0.3}, from=3-5, to=4-6]
	\arrow["f"'{pos=0.3}, from=3-7, to=4-8]
	\arrow[Rightarrow, from=2-4, to=3-3]
	\arrow[Rightarrow, from=2-6, to=3-5]
	\arrow[Rightarrow, from=2-8, to=3-7]
	\arrow[from=1-11, to=2-12]
	\arrow[from=1-11, to=3-11]
	\arrow[from=3-11, to=4-12]
	\arrow[from=2-12, to=4-12]
	\arrow[Rightarrow, from=2-12, to=3-11]
	\arrow[shorten >=3pt, no head, from=1-7, to=0]
\end{tikzcd}\]
To finish the proof, we just need to check that the colimit square is a minimal model of $f$. The algebras $\cM_\infty, \cN_\infty$ are minimal since they are both unions of minimal sub-algebras $$
\cM_\infty=\bigcup_{k\geq 1} \cM_k \:\:\:\:\:\:\:\:\:\:\:\:\:\:\:\:\:\:\:\:
\cN_\infty=\bigcup_{k\geq 1} \cN_k.
$$
It thus suffices to check that the maps $\cM_\infty \to \cA$ and $\cN_\infty \to \cB$ are quasi-isomorphisms. We treat the case of $\cM_\infty$, that of $\cN_\infty$ being similar. Denoting $C_k=\Cone(\cM_k \to \cA)$, we can write 
$$
H^*C_\infty = \colim_k \: H^* C_k = 0
$$
since filtered colimits are exact in $\Ch_\Q$, and each $C_k$ is $k$-connected. It follows that $\cM_\infty \weak \cA$ is a quasi-isomorphism.
\end{proof}

\section{Persistent minimal models}

In this section, we generalize our construction for a single map to the persistent setting. We formulate this a series of push-outs along \textit{interval spheres}. We develop a procedure called \textit{interval surgery} -- the attachment of intervals into persistent CDGAs to correct errors in persistent cohomology approximation. 

\subsection{Interval spheres} 
A \textit{persistent cochain complex} is a persistent object in the category $\Ch_\Q$ of complexes, i.e., a functor $\mathbb{X} : [n] \to \mathsf{Ch}_\Q^*$. 
\begin{notation} Let $\X \in p\Ch_\Q$ be a persistent cochain complex and $k \in \N$.
\begin{itemize}
    \item $\X^k, Z\X^k \in p\sVec$ are the persistence modules of $k$-cochains and $k$-cocycles respectively.
    \item $H^k \X \in p\sVec$ is the persistent $k$-th cohomology of $\X$.
    \item If $x_s \in \X^k(s)$ then $x_t := \X^k(s \leq t)(x_s)$ for $s <t$.
\end{itemize}
\end{notation}

\subsubsection{Interval spheres and disks} In this section we introduce interval spheres and disks, as first defined in \cite{giunti_2021}. 
\begin{definition}[Interval spheres]\label{spheres} 
    Let $0 \leq s < t \leq n+1$. The \textit{interval sphere} in degree $k$ is the functor $\Sp_{[s,t)}^k : [n] \to \mathsf{Vec}_\Q$ defined by 
    \[
    \Sp_{[s,t)}^k(p) = \begin{cases}
        0 & p < s\\
        S^k & s \leq p < t\\
        D^k & t \leq p
    \end{cases}
    \] where all structure maps are identities except for 
    \[
     \Sp_{[s,t)}^k(s-1<s) : 0 \hookrightarrow S^k \; \; \; \; \; \; \text{and} \; \; \; \; \; \; \Sp_{[s,t)}^k(t-1<t) : S^k \hookrightarrow D^k.
    \] The \textit{persistent disk} $\Di_{s}^k$ is the special case where we set $s = t$.
\end{definition} 

Since any persistent module $\mathbb{V} : [n] \to \mathsf{Vec}_\Q$ admits an interval decomposition
    $
    \V = \bigoplus_{\alpha} \bI_{[s_\alpha,t_\alpha)}
    $ we can define the interval sphere over $\mathbb{V}$ as the direct sum
    $\Sp^k(\V)=\bigoplus_{\alpha} \Sp^k_{[s_\alpha, t_\alpha)}$ and similarly for persistent disks. This is well-defined up-to re-labelling, and there a natural inclusion $\Sp(\V)\subseteq \Di(\V)$ induced by those of \Cref{spheres}.

Post-composition with the free commutative graded algebra functor $\Lambda : \mathsf{Ch}_\Q^* \to \CDGA_\Q$ yields the CDGA interval spheres and disks
\[
\Lambda \Sp_{[s,t)}^k := \Lambda \circ \Sp_{[s,t)}^k \: \: \: \: \text{and} \: \: \: \Lambda \Di_{s}^k := \Lambda \circ \Di_{s}^k.
\] From \cite[4.5]{giunti_2021}, there is an isomorphism
$$\Hom(\Di_s^k, \X)\cong \X^{k-1}(s) \:\:\:\:\:\:\:\:\:\: \Hom(\Sp_{[s,t)}^k, \X) \cong Z\X^k(s)\times_{Z\X^k(t)}\X^{k-1}(t)
$$
valid for any persistent cochain complex $\X$ and for each $k\geq 1$. We can thus denote morphisms
% https://q.uiver.app/#q=WzAsNCxbMCwwLCJcXFNwXntrKzF9X3tbcyx0KX0iXSxbMiwwLCJcXGJBIl0sWzQsMCwiXFxEaV57aysxfV9zIl0sWzYsMCwiXFxiQiJdLFswLDEsIih4X3MseV90KSJdLFsyLDMsInpfcyJdXQ==
\[\begin{tikzcd}
	{\Sp^{k+1}_{[s,t)}} && \bA && {\Di^{k+1}_s} && \bB
	\arrow["{(x_s,y_t)}", from=1-1, to=1-3]
	\arrow["{z_s}", from=1-5, to=1-7]
\end{tikzcd}\] by elements, where $dy_t = x_t \in \bA^{k+1}(t)$ and $z_s \in \bB^k(s)$. These observations hold equally well for CDGA interval spheres and disks. 

\subsection{Persistent Hirsch Extensions} Our goal is to generalise the connections among Hirsch extensions, minimal models and approximation telescopes to the persistent setting. To this end, we propose the following definition of a \textit{persistent} Hirsch extension. 

\begin{definition} \label{persistent-hirsch-extension}
    Let $\bA$ be a persistent CDGA, $\V$ a graded persistence module of finite type. A \textit{Hirsch extension} of $\bA$ by $\V$ is a homomorphism $\bA \to \bB$ of persistent CDGAs that fits into a pushout diagram
    \begin{center}
    \begin{tikzcd}
        \Lambda \Sp(\V) \ar[r] \ar[d] & \bA \ar[d] \\
        \Lambda \Di(\V) \ar[r] & \bB \arrow[ul, phantom, "\ulcorner", very near start]
    \end{tikzcd}
    \end{center}
    where the map $\Lambda \Sp(\V) \to \Lambda \Di(\V)$ is induced by the inclusion $\Sp(\V) \subseteq \Di(\V)$ of \Cref{spheres}. In this situation, we write $\bB=\bA \langle \V \rangle$, leaving implicit the data of the attaching map $\Lambda \Sp(\V)\to \bA$. We say that the Hirsch extension is \textit{of degree $k$} when $\V$ is concentrated in degree $k+1$.
\end{definition}

The following statement relates persistent Hirsch extensions to the usual notions of Hirsch extension of CDGAs. 

\begin{proposition}\label{pers_hirsch_structure}
    Let $\bA$ be a persistent CDGA and $\V$ a graded persistence module. Consider a persistent Hirsch extension $\bA \to \bA\langle \V \rangle$ of $\bA$ by $\V$. 
    \begin{enumerate}
        \item For each index $t$, the algebra homomorphism $\bA(t) \to \bA\langle \V \rangle (t)$ is a Hirsch extension of $\bA(t)$ by $\V(t)$.
        \item For all indices $s< t$, the homomorphism $\bA\langle \V \rangle(s<t)$ is a Hirsch extension of $\bA(s<t)$ by $\V(s) \to \V(t)$.
    \end{enumerate}
\end{proposition}
\begin{proof}
    Note that if $\V$ is concentrated in degree 0, then $\bA \langle \V \rangle =0$. If $\V=\bigoplus_{n> 0} \V_n$ is of finite type, there is a decomposition $\bA \langle \V \rangle = \bigotimes_{n\geq 0} \bA \langle \V_n \rangle$, so that we can reduce to proving the statement when $\V$ is concentrated in a single degree $n>0$. In this case, if $\V=\bigoplus_\alpha \I^n_{[s_\alpha,t_\alpha)}$ is an interval decomposition of $\V$, we similarly find $\bA\langle \V \rangle = \bigotimes_\alpha \bA\langle \I^n_{[s_\alpha, t_\alpha)} \rangle$, so that it suffices to prove the statement when $\V=\I_{[s,t)}^n$ is an interval module. The result now follows from the lemma.
\end{proof}
\begin{lemma} \label{interval-attaching-hirsch}
    A Hirsch extension $\bA \to \bA \langle \I_{[s,t)}^{k} \rangle$ of $\bA$ by an interval module $\I_{[s,t)}^k$ is determined by the choice of a pair $(x_t, y_s) \in \bA^{k-1}(t) \times \bA^{k}(s)$ satisfying $dx_t=y_t$. In this situation, the components of $\bA \langle \I_{[s,t)}^{k} \rangle$ are described as follows:
    $$
    \bA \langle \I_{[s,t)}^{k} \rangle (u) \cong 
    \begin{cases}
        \bA(u)\langle \: \gamma_u \: | \: d\gamma_u=y_u \: \rangle 
        & \text{when $u\in [s,t)$},\\
        \bA(u) & \text{otherwise.}
    \end{cases}
    $$
    Moreover, the structure maps of $\bA \langle \I_{[s,t)}^{k} \rangle$ have the following form:
    $$
    \bA \langle \I_{[s,t)}^{k} \rangle(u<v) \cong 
    \begin{cases}
        \bA(u) \xrightarrow{u<v} \bA(v) \hookrightarrow \bA(v) \langle \gamma_v\rangle
        & \text{when $u<s$ and $v\in [s,t)$},\\
        \bA(u)\langle x_u \rangle \xrightarrow[\gamma_u \mapsto x_v]{} \bA(u) \xrightarrow[u<v]{} \bA(v)
        & \text{when $u\in [s,t)$ and $v\geq t$},
    \end{cases}
    $$
    and isomorphisms otherwise.
\end{lemma}
\begin{proof}
    A Hirsch extension $\bA \to \bA \langle \I_{[s,t)}^{k} \rangle$ is fully determined by the choice of a map $\Lambda \Sp_{[s,t)}^k \to \bA$, since $\Sp_{[s,t)}^k=\Sp(\I_{[s,t)}^k)$. Such a map is in turn characterized by a pair $(x_t, y_s)$ as above. Since pushouts in persistent objects can be computed pointwise, the following square
    $$
    \begin{tikzcd}
        \Lambda \Sp_{[s,t)}^k(u) \ar[r] \ar[d] & \bA(u) \ar[d] \\
        \Lambda \Di_s^k(u) \ar[r, "\gamma_u"] & \bA \langle \I_{[s,t)}^{k} \rangle (u) \arrow[ul, phantom, "\ulcorner", very near start]
    \end{tikzcd}
    $$
    is a pushout diagram which is natural in the time index $u$. We thus obtain natural isomorphisms $\bA \langle \I_{[s,t)}^{k} \rangle (u) \cong \bA(u) \langle \I_{[s,t)}^{k}(u) \rangle$. Naturality ensures that, given time indices $u<s\leq v < t \leq w$, the corresponding structure maps fit into the diagram
    $$
    \begin{tikzcd}
        \bA(u) \ar[d, equals]\ar[r] 
        & \bA(v) \ar[d]\ar[r] 
        & \bA(w) \ar[d, equals] \\
        \bA(u) \ar[r] 
        & \bA(v)\langle \gamma_v \rangle \ar[r, "\gamma_v \mapsto x_w"'] 
        & \bA(w).
    \end{tikzcd}
    $$
    This concludes the proof.
\end{proof}

The next statement is a persistent analogue of the fact that an ordinary Hirsch extension $\cA \to \cA\langle V \rangle$ factors as a (possibly transfinite) composite of elementary extensions $\cA \to \cA\langle v_i \rangle$, indexed by a basis $\{v_i\}$ of $V$. For the sake of simplicity, we limit ourselves to the case of persistence modules of finite type.

\begin{proposition}
    Let $\bA$ be a connected persistent CDGA, and $\V$ a graded persistence module of finite type. Each persistent Hirsch extension $\bA \to \bA\langle \V \rangle$ factors as a finite sequence of Hirsch extensions along interval modules $\I_{[s,t)}^{k}$.
\end{proposition}
\begin{proof}
    Given a decomposition $\W=\V \bigoplus \I_{[s,t)}^k$, the inclusion $\Sp(\W) \subseteq \Di(\W)$ factors as 
    $$
    \Sp(\W)=\Sp(\V)\oplus \Sp_{[s,t)}^k \subseteq \Sp(\V) \oplus \Di_{s}^k \subseteq \Di(\V)\oplus \Di_s^k = \Di(\W). 
    $$
    As a result, any Hirsch extension $\bA \to \bA\langle \W \rangle$ factors as a composite of Hirsch extensions $\bA \to \bA \langle \V \rangle \to \bA \langle \V \rangle \langle \I_{[s,t)}^k \rangle$. The result follows from an easy induction.
\end{proof}

\subsection{Persistent Minimal Models} The use of a finite indexing set $[n]$ allows us to avoid complicated details relating to homotopy coherence in persistent minimal models over arbitrary linear posets. The definition we give here specifies homotopies only between neighbouring indices, and is designed so that the issue of homotopy coherence is considerably simplified.

\begin{definition} \label{persistent-minimal-model-definition}
    Let $\bA$ be a persistent CDGA. A \textit{$k$-minimal model} $m : \bM \xrightarrow{k}_2 \bA$ for $\bA$ is the data of:
   % https://q.uiver.app/#q=WzAsMTAsWzIsMCwiXFxtTShyKSJdLFs0LDAsIlxcbU0ocisxKSJdLFsyLDIsIlxcbUEocikiXSxbNCwyLCJcXG1BKHIrMSkiXSxbMCwwLCJcXGxkb3RzIl0sWzYsMCwiXFxtTShyKzIpIl0sWzYsMiwiXFxtQShyKzIpIl0sWzAsMiwiXFxsZG90cyJdLFs4LDAsIlxcbGRvdHMiXSxbOCwyLCJcXGxkb3RzIl0sWzAsMV0sWzAsMiwibShyKSJdLFsyLDNdLFsxLDMsIm0ocisxKSJdLFs0LDBdLFsxLDVdLFsxLDIsIkgocikiLDAseyJzaG9ydGVuIjp7InNvdXJjZSI6MjAsInRhcmdldCI6MjB9LCJsZXZlbCI6Mn1dLFs1LDYsIm0ocisyKSJdLFszLDZdLFs1LDMsIkgocisxKSIsMCx7InNob3J0ZW4iOnsic291cmNlIjoyMCwidGFyZ2V0IjoyMH0sImxldmVsIjoyfV0sWzcsMl0sWzUsOF0sWzYsOV1d
\begin{equation} \label{tame-minimal-model} \begin{tikzcd}[sep=scriptsize]
	\ldots && {\bM(r)} && {\bM(r+1)} && {\bM(r+2)} && \ldots \\
	\\
	\ldots && {\bA(r)} && {\bA(r+1)} && {\bA(r+2)} && \ldots
	\arrow[from=1-1, to=1-3]
	\arrow[from=1-3, to=1-5]
	\arrow["{m(r)}", from=1-3, to=3-3]
	\arrow[from=1-5, to=1-7]
	\arrow["{H(r)}", shorten <=13pt, shorten >=13pt, Rightarrow, from=1-5, to=3-3]
	\arrow["{m(r+1)}", from=1-5, to=3-5]
	\arrow[from=1-7, to=1-9]
	\arrow["{H(r+1)}", shorten <=14pt, shorten >=14pt, Rightarrow, from=1-7, to=3-5]
	\arrow["{m(r+2)}", from=1-7, to=3-7]
	\arrow[from=3-1, to=3-3]
	\arrow[from=3-3, to=3-5]
	\arrow[from=3-5, to=3-7]
	\arrow[from=3-7, to=3-9]
\end{tikzcd}\end{equation} such that each $m(r) : \bM(r) \to \A(r)$ is a $k$-minimal model.
\end{definition}

\begin{remark}
    The definition above describes a pseudo-natural transformation from $[n]$ with the trivial 2-categorical structure into CDGA with the 2-category structure induced by its simplicial enrichment, motivating the notation $m : \bM \to_2 \bA$.
\end{remark}

 The diagram \ref{tame-minimal-model} is a \textit{minimal model} when each $m(r) : \bM(r) \to \bA(r)$ is a minimal model. We will refer to the object $\bM$ as the minimal model for $\bA$, with the maps and homotopies defined implicitly. Since each $\bM(r)$ is cofibrant in the usual model structure on $\CDGA_\Q$, homotopy is an equivalence relation, implying that \textit{all} of the structure maps in $\bM$ are Sullivan representatives for the corresponding ones in $\bA$.

\subsection{Main construction} 
\subsubsection{Cones} For persistent CDGAs, specifying \textit{explicit choices} of the homotopies permits the construction of a functorial cone. The key is that homotopies are encoded into the structure maps. Each square in the persistent minimal model
\begin{center}
        \begin{tikzcd}[row sep = 3em, column sep = 4em]
            \bM(r) \ar[r, "\bM(r < r+1)"] \ar[d, "m(r)"'] & \bM(r+1) \ar[d, "m(r+1)"] \ar[dl, Rightarrow, "H(r)", shorten = 1.0ex]\\
            \bA(r) \ar[r, "\bA(r < r+1)"'] & \bA(r+1)
        \end{tikzcd}
\end{center} induces a map of cochain cones given by
\begin{align*} 
\varphi(r) : C_{m(r)} & \to C_{m(r+1)}\\
(v_r,a_r) & \mapsto \Big(v_{r+1}, a_{r+1} + \int_0^1 H(r)(v_r) \Big)
\end{align*}
where $v_{r+1} = \bM(r<r+1)(v_r)$ and $a_{r+1} = \bA(r<r+1)(a_r)$. 

\begin{definition}
    Let $\bM$ be a $k$-minimal model for $\bA \in p\CDGA_\Q$, with maps and homotopies $\{ m(r), H(r) \}$. The \textit{cone} $\C_m : [n] \to \Ch_\Q$ of $m$ is the functor generated by the sequence $$C_{m(0)} \xrightarrow{\varphi(0)} C_{m(1)} \xrightarrow{\varphi(1)} \ldots \xrightarrow{\varphi(n-1)} C_{m(n)}.$$
\end{definition} 

\begin{comment}
\begin{remark}
    The tame cone is defined as such to avoid homotopy coherence questions and simplify induction arguments. To formalise the cone away from the tame setting, one needs an appropriate notion of homotopy commutative morphism from $\bM$. 
\end{remark}
\end{comment}

\subsubsection{Interval surgery} The classical approximation telescope (\ref{pointwise_induction_lemma}) successively corrects defects in a minimal model by attaching in the cohomological generators of the cone.  In the persistent setting, the cohomological generators of the cone are intervals. Persistent Hirsch extensions then permit an 'interval surgery', where attaching in such intervals produces an extension of a $(k-1)$-minimal model to a $k$-minimal model.

\begin{lemma}[Interval surgery] \label{persistent-inductive-cone} Let $\bA \in p\CDGA_\Q$ be simply-connected and $m : \bM \to_2 \bA$ be a $(k-1)$-minimal model with maps and homotopies $\{ m(r), H(r) \mid r \in [n] \}$ and cone $\C_m$. There exists an extension to a tame $k$-minimal model $\overline{m} :\overline{\bM} \to_2 \bA$ fitting into the following diagram.
% https://q.uiver.app/#q=WzAsNSxbMCwwLCJcXExhbWJkYSBcXGJTXntrKzF9KEhea1xcQ19tKSJdLFswLDIsIlxcTGFtYmRhIFxcYkRee2srMX0oSF5rXFxDX20pIl0sWzIsMCwiXFxiTSJdLFsyLDIsIlxcb3ZlcmxpbmV7XFxiTX0iXSxbNCwwLCJcXGJBIl0sWzAsMSwiIiwwLHsic3R5bGUiOnsidGFpbCI6eyJuYW1lIjoiaG9vayIsInNpZGUiOiJ0b3AifX19XSxbMCwyXSxbMSwzXSxbMiwzLCIiLDIseyJzdHlsZSI6eyJ0YWlsIjp7Im5hbWUiOiJob29rIiwic2lkZSI6InRvcCJ9fX1dLFsyLDQsIm0iXSxbMiw0LCIyIiwyLHsibGFiZWxfcG9zaXRpb24iOjgwfV0sWzMsNCwiXFxvdmVybGluZXttfSIsMl0sWzMsNCwiMiIsMix7ImxhYmVsX3Bvc2l0aW9uIjo5MH1dXQ==
\[\begin{tikzcd}[sep = scriptsize]
	{\Lambda \Sp^{k+1}(H^k\C_m)} && \bM && \bA \\
	\\
	{\Lambda \Di^{k+1}(H^k\C_m)} && {\overline{\bM}}
	\arrow[from=1-1, to=1-3]
	\arrow[hook, from=1-1, to=3-1]
	\arrow["m", from=1-3, to=1-5]
	\arrow["2"'{pos=0.8}, from=1-3, to=1-5]
	\arrow[hook, from=1-3, to=3-3]
	\arrow[from=3-1, to=3-3]
	\arrow["{\overline{m}}"', from=3-3, to=1-5]
	\arrow["2"'{pos=0.9}, from=3-3, to=1-5]
\end{tikzcd}\]
\end{lemma}

\begin{remark}
    "Commutativity" of the right-most triangle is used to signify that new homotopies between neighbouring indices are extensions of old ones. 
\end{remark}

\begin{proof}
    Pick an interval decomposition $H^k\C_m \cong \bigoplus \bI_{[p,q)}$. We will attach each interval independently, making the requisite extensions of $m$ and the homotopies $H$ over the new elements. Choose a persistent cocycle representative
    % https://q.uiver.app/#q=WzAsNCxbMCwwLCIodl9wLGFfcCkgPSAodl9wLFxcdGlsZGV7YV9wfSkiXSxbMiwwLCJcXEJpZyh2X3twKzF9LGFfe3ArMX0gKyBcXGludF8wXjFIX3Aodl9wKVxcQmlnKT0odl97cCsxfSxcXHRpbGRle2F9X3twKzF9KSJdLFs0LDAsIlxcY2RvdHMiXSxbNSwwLCIodl9xLCBcXHRpbGRle2F9X3EpIl0sWzAsMSwiXFx2YXJwaGkocCkiXSxbMSwyLCJcXHZhcnBoaShzKzEpIl0sWzIsM11d
\[\begin{tikzcd}[sep = scriptsize]
	{(v_p,a_p) = (v_p,\tilde{a}_p)} && {\Big(v_{p+1},a_{p+1} + \int_0^1H(p)(v_p)\Big)=(v_{p+1},\tilde{a}_{p+1})} && \cdots & {(v_q, \tilde{a}_q)}
	\arrow["{\varphi(p)}", from=1-1, to=1-3]
	\arrow["{\varphi(p+1)}", from=1-3, to=1-5]
	\arrow[from=1-5, to=1-6]
\end{tikzcd}\] for $\bI_{[p,q)}$ in the $k$-th tame cone cohomology, where we denote $\tilde{a}_{r+1} = \pi_\mathbb{A}\varphi(r)(v_r,\tilde{a}_r)$ for the projection $\pi_\mathbb{A}: \mathbb{C}_{m} \to \mathbb{A}$. Let $d(u_t,b_t) = (v_t,a_t)$ witness exactness in the cone at time $t$. Make an interval extension
% https://q.uiver.app/#q=WzAsNCxbMCwwLCJcXExhbWJkYSBcXFNwXntrKzF9X3tbcCxxKX0iXSxbMiwwLCJcXGJNIl0sWzAsMiwiXFxMYW1iZGEgXFxEaV57aysxfV97cH0iXSxbMiwyLCJcXGJNJyJdLFswLDEsIih2X3AseV9xKSJdLFswLDJdLFsxLDNdLFsyLDMsIlxcZ2FtbWFfcCIsMl0sWzMsMCwiIiwxLHsic3R5bGUiOnsibmFtZSI6ImNvcm5lciJ9fV1d
\[\begin{tikzcd}[sep = scriptsize]
	{\Lambda \Sp^{k+1}_{[p,q)}} && \bM \\
	\\
	{\Lambda \Di^{k+1}_{p}} && {\bM'.}
	\arrow["{(v_p,u_q)}", from=1-1, to=1-3]
	\arrow[from=1-1, to=3-1]
	\arrow[from=1-3, to=3-3]
	\arrow["{\gamma_p}"', from=3-1, to=3-3]
	\arrow["\lrcorner"{anchor=center, pos=0.125, rotate=180}, draw=none, from=3-3, to=1-1]
\end{tikzcd}\] of $\bM$ along $\Lambda \bI_{[p,q)}$. Using the formulas of \Cref{inductivestep}, define extensions $\overline{m}(r)(\gamma_r) = \tilde{a}_r$ for all $p \leq r < q$ and 
$$\overline{H}(r)(\gamma_r) = \begin{cases} \bA(r<r+1)\tilde{a}_r + \int_0^t H(r)(v_r) & p \leq r < q-1\\
    \bA(q-1<q)(\tilde{a}_{q-1}) + \int_0^t H(q-1)(v_{q-1}) + d(y_q \otimes t) & r = q-1 \end{cases}$$ which are well-defined CDGA maps by the same arguments. Concurrently attaching all intervals
% https://q.uiver.app/#q=WzAsNSxbMCwwLCJcXExhbWJkYSBcXGJTXntrKzF9KEhea1xcQ19tKSA9IFxcYmlnb3RpbWVzIFxcTGFtYmRhIFxcU3Bee2srMX1fe1twLHEpfSJdLFswLDIsIlxcTGFtYmRhIFxcYkRee2srMX0oSF5rXFxDX20pID0gXFxiaWdvdGltZXMgXFxMYW1iZGEgXFxEaV57aysxfV9wIl0sWzIsMCwiXFxiTSJdLFsyLDIsIlxcb3ZlcmxpbmV7XFxiTX0iXSxbNCwwLCJcXGJBIl0sWzAsMSwiIiwwLHsic3R5bGUiOnsidGFpbCI6eyJuYW1lIjoiaG9vayIsInNpZGUiOiJ0b3AifX19XSxbMCwyLCJcXG90aW1lcyh4X3AsdV9xKSJdLFsxLDMsIlxcb3RpbWVzIFxcZ2FtbWFfcCJdLFsyLDMsIiIsMix7InN0eWxlIjp7InRhaWwiOnsibmFtZSI6Imhvb2siLCJzaWRlIjoidG9wIn19fV0sWzIsNCwibSJdLFsyLDQsIjIiLDIseyJsYWJlbF9wb3NpdGlvbiI6ODB9XSxbMywwLCIiLDAseyJzdHlsZSI6eyJuYW1lIjoiY29ybmVyIn19XSxbMyw0LCIyIiwyLHsibGFiZWxfcG9zaXRpb24iOjkwfV0sWzMsNCwiXFxvdmVybGluZXttfSIsMl1d
\[\begin{tikzcd}[row sep = scriptsize]
	{\Lambda \Sp^{k+1}(H^k\C_m) = \bigotimes \Lambda \Sp^{k+1}_{[p,q)}} && \bM && \bA \\
	\\
	{\Lambda \Di^{k+1}(H^k\C_m) = \bigotimes \Lambda \Di^{k+1}_p} && {\overline{\bM}}
	\arrow["{\otimes(v_p,u_q)}", from=1-1, to=1-3]
	\arrow[hook, from=1-1, to=3-1]
	\arrow["m", from=1-3, to=1-5]
	\arrow["2"'{pos=0.8}, from=1-3, to=1-5]
	\arrow[hook, from=1-3, to=3-3]
	\arrow["{\otimes \gamma_p}", from=3-1, to=3-3]
	\arrow["\lrcorner"{anchor=center, pos=0.125, rotate=180}, draw=none, from=3-3, to=1-1]
	\arrow["2"'{pos=0.9}, from=3-3, to=1-5]
	\arrow["{\overline{m}}"', from=3-3, to=1-5]
\end{tikzcd}\]
    thus yields a well-defined extension of $m$ to $\overline{m} : \overline{\bM} \to_2 \bA$. To see that this is a $k$-minimal model, note that it restricts exactly to the construction of minimal models of maps between neighbouring discretization indices from \Cref{inductivestep} 
\end{proof}

\begin{theorem}[Structure Theorem]\label{structure-theorem}
    Every  simply-connected, finite-type persistent CDGA $\bA$ admits a persistent minimal model $\bM$ whose skeletal filtration is constructed as a sequence of degree $k$ persistent Hirsch extensions
    % https://q.uiver.app/#q=WzAsNSxbMCwwLCJcXExhbWJkYSBcXGJTXntrKzF9KEhea1xcQ19tKSJdLFswLDIsIlxcTGFtYmRhIFxcYkRee2srMX0oSF5rXFxDX20pIl0sWzIsMCwiXFxiTV97ay0xfSJdLFsyLDIsIlxcYk1fayJdLFs0LDAsIlxcYkEiXSxbMCwxLCIiLDAseyJzdHlsZSI6eyJ0YWlsIjp7Im5hbWUiOiJob29rIiwic2lkZSI6InRvcCJ9fX1dLFswLDJdLFsxLDNdLFsyLDMsIiIsMix7InN0eWxlIjp7InRhaWwiOnsibmFtZSI6Imhvb2siLCJzaWRlIjoidG9wIn19fV0sWzIsNCwibSJdLFszLDQsIm1fayIsMix7ImN1cnZlIjozfV0sWzIsNCwiMiIsMix7ImxhYmVsX3Bvc2l0aW9uIjo4MH1dLFszLDQsIjIiLDIseyJsYWJlbF9wb3NpdGlvbiI6OTAsImN1cnZlIjozfV1d
\begin{equation} \label{main-cofiber-sequence-diagram} \begin{tikzcd}[sep=scriptsize]
	{\Lambda \Sp^{k+1}(H^k\C_m)} && {\bM_{k-1}} && \bA \\
	\\
	{\Lambda \Di^{k+1}(H^k\C_m)} && {\bM_k}
	\arrow[from=1-1, to=1-3]
	\arrow[hook, from=1-1, to=3-1]
	\arrow["m_{k-1}", from=1-3, to=1-5]
	\arrow["2"'{pos=0.8}, from=1-3, to=1-5]
	\arrow[hook, from=1-3, to=3-3]
	\arrow[from=3-1, to=3-3]
	\arrow["{m_k}"', from=3-3, to=1-5]
	\arrow["2"'{pos=0.9}, from=3-3, to=1-5]
\end{tikzcd}\end{equation} where $\C_{m_{k-1}}$ is the tame cone at degree $(k-1)$.
\end{theorem}

This generalizes the interval decomposition, where the direct sum is replaced by a twisted coproduct of cells. Combining this theorem with the structure of persistent Hirsch extensions (\ref{pers_hirsch_structure} and \ref{interval-attaching-hirsch}), the persistent minimal model is finitely encoded by two data at each level $k$.
\begin{enumerate}
    \item The interval summands $\bI_{[s,t)}$ of $H^k\C_{m_{k-1}}$.
    \item The attaching maps $$\Lambda \Sp^{k+1}_{[p,q)} \xrightarrow{(v_p,u_q)} \bM_{k-1}.$$
\end{enumerate} Together these data specify a surprisingly concise presentation
% https://q.uiver.app/#q=WzAsOCxbMCwwLCJcXGJNID0gXFxMYW1iZGEoXFxtYXRoc2Nye0d9IFxcbWlkIFxcbWF0aHNjcntSfSkiXSxbMiwwLCJcXEJpZ1xceyBcXCwgXFwsIFxcZ2FtbWFfcCBcXCwgXFwsIFxcQmlnXFx9Il0sWzIsMSwiXFx0ZXh0e0dlbmVyYXRvcnMgfSBcXG1hdGhzY3J7R30iXSxbNCwwLCJcXEJpZ2dcXHsgXFwsIFxcLCBkXFxnYW1tYV9wID0gdl9wLCBcXGdhbW1hX3EgPSB1X3EgXFwsIFxcLCBcXEJpZ2dcXH0iXSxbNCwxLCJcXHRleHR7UmVsYXRpb25zIH0gXFxtYXRoc2Nye1J9Il0sWzYsMCwiXFxiSV97W3AscSl9IFxcaG9va3JpZ2h0YXJyb3cgSF5rIFxcQ197bV97ay0xfX0iXSxbMCwxLCJcXHRleHR7UHJlc2VudGF0aW9ufSJdLFs2LDEsIlxcdGV4dHtJbnRlcnZhbHN9Il1d
\begin{equation} \label{presentation} \begin{tikzcd}[sep=tiny]
	{\bM = p\Lambda(\mathscr{G} \mid \mathscr{R})} && {\Big\{ \, \, \gamma_p \in \bM^k(p) \, \, \Big\}} && {\Bigg\{ \, \, d\gamma_p = v_p, \gamma_q = u_q \, \, \Bigg\}} && {\bI_{[p,q)} \hookrightarrow H^k \C_{m_{k-1}}} \\
	{\text{Presentation}} && {\text{Generators } \mathscr{G}} && {\text{Relations } \mathscr{R}} && {\text{Intervals}}
\end{tikzcd} \end{equation} of the persistent homotopy type of $\bA$.

\section{Persistent rational homotopy theory}

In the previous section, we put special emphasis on appending intervals into the space via attachments along CDGA interval spheres. This approach allows us to generalize the classical correspondence (\ref{classical-postnikov}) between CDGA cell attachment and Postnikov towers to the persistent setting.

\subsection{Persistent cells and rational homotopy groups} In this section, let $p^*\sSet := \Fun([n]^{op}, \sSet)$ be the category of copersistent simplicial sets over $[n]$.

\subsubsection{Eilenberg-Maclane Spaces} Let $\V$ be a compact, copersistent graded rational vector space concentrated in degree $k$. A \textit{copersistent Eilenberg Maclane space} over $\V$ is a copersistent simplicial set $K(\V, k) \in p^*\sSet$ such that 
$$\pi_* K(\V,k) = \V.$$ Since homotopy groups commute with products naturally we have $$\V^\vee \cong \bigoplus \bI_{[s_\gamma,t_\gamma)} \hspace{2em} \Rightarrow \hspace{2em} \prod_\gamma K(\bI_{[s_\gamma,t_\gamma)}, k) \text{ is a } K(\V,k).$$ Applying \Cref{classical-postnikov} pointwise and examining the structure maps, a basic observation is that
$$\langle \Lambda \bS^k(\V) \rangle \text{ is a } K(\V,k),$$ i.e., that interval spheres are algebraic models of copersistent Eilenberg-Maclane spaces. 

\subsubsection{Postnikov Towers} Let $\X$ be a simply-connected copersistent simplicial set. A tame minimal model
% https://q.uiver.app/#q=WzAsMTAsWzIsMCwiXFxtYXRoYmJ7TX1faSJdLFsyLDIsIlxcbWF0aGNhbHtBfV97UEx9KFxcbWF0aGJie1h9X2kpIl0sWzQsMiwiXFxtYXRoY2Fse0F9X3tQTH0oXFxtYXRoYmJ7WH1fe2krMX0pIl0sWzQsMCwiXFxtYXRoYmJ7TX1fe2krMX0iXSxbNiwwLCJcXG1hdGhiYntNfV97aSsyfSJdLFs4LDAsIlxcY2RvdHMiXSxbNiwyLCJcXG1hdGhjYWx7QX1fe1BMfShcXG1hdGhiYntYfV97aSsyfSkiXSxbMCwwLCJcXGNkb3RzIl0sWzAsMiwiXFxjZG90cyJdLFs4LDIsIlxcY2RvdHMgIl0sWzEsMl0sWzAsM10sWzQsNV0sWzIsNl0sWzMsMiwiXFxzaW1lcSJdLFswLDEsIlxcc2ltZXEiXSxbNywwXSxbOCwxXSxbMyw0XSxbNCw2LCJcXHNpbWVxIl0sWzIsNCwiSF97aSsxfSIsMix7ImxldmVsIjoyfV0sWzEsMywiSF9pIiwyLHsibGV2ZWwiOjJ9XSxbNiw5XV0=
\[\begin{tikzcd}[column sep=tiny,row sep=small]
	\cdots && {\mathbb{M}_i} && {\mathbb{M}_{i+1}} && {\mathbb{M}_{i+2}} && \cdots \\
	\\
	\cdots && {\mathcal{A}_{PL}(\mathbb{X}_i)} && {\mathcal{A}_{PL}(\mathbb{X}_{i+1})} && {\mathcal{A}_{PL}(\mathbb{X}_{i+2})} && {\cdots }
	\arrow[from=1-1, to=1-3]
	\arrow[from=1-3, to=1-5]
	\arrow["\simeq", from=1-3, to=3-3]
	\arrow[from=1-5, to=1-7]
	\arrow["\simeq", from=1-5, to=3-5]
	\arrow[from=1-7, to=1-9]
	\arrow["\simeq", from=1-7, to=3-7]
	\arrow[from=3-1, to=3-3]
	\arrow["{H_i}"', Rightarrow, from=3-3, to=1-5]
	\arrow[from=3-3, to=3-5]
	\arrow["{H_{i+1}}"', Rightarrow, from=3-5, to=1-7]
	\arrow[from=3-5, to=3-7]
	\arrow[from=3-7, to=3-9]
\end{tikzcd}\] of $\X$ yields an adjoint diagram:
% https://q.uiver.app/#q=WzAsMTAsWzIsMCwiXFxsYW5nbGUgXFxtYXRoYmJ7TX1faSBcXHJhbmdsZSJdLFsyLDIsIlxcbWF0aGJie1h9X2kiXSxbNCwyLCJcXG1hdGhiYntYfV97aSsxfSJdLFs0LDAsIlxcbGFuZ2xlIFxcbWF0aGJie019X3tpKzF9IFxccmFuZ2xlIl0sWzYsMCwiXFxsYW5nbGUgXFxtYXRoYmJ7TX1fe2krMn0gXFxyYW5nbGUiXSxbOCwwLCJcXGNkb3RzIl0sWzYsMiwiXFxtYXRoY2Fse0F9X3tQTH0oXFxtYXRoYmJ7WH1fe2krMn0pIl0sWzAsMCwiXFxjZG90cyJdLFswLDIsIlxcY2RvdHMiXSxbOCwyLCJcXGNkb3RzICJdLFsyLDFdLFszLDBdLFs1LDRdLFs2LDJdLFsyLDMsIlxcc2ltZXEiLDJdLFsxLDAsIlxcc2ltZXEiLDJdLFswLDddLFsxLDhdLFs0LDNdLFs2LDQsIlxcc2ltZXEiLDJdLFs0LDIsIlxcbGFuZ2xlIEhfe2krMX0gXFxyYW5nbGUgIiwwLHsibGV2ZWwiOjJ9XSxbMywxLCJcXGxhbmdsZSBIX2kgXFxyYW5nbGUiLDAseyJsZXZlbCI6Mn1dLFs5LDZdXQ==
\[\begin{tikzcd}[sep=small]
	\cdots && {\langle \mathbb{M}_i \rangle} && {\langle \mathbb{M}_{i+1} \rangle} && {\langle \mathbb{M}_{i+2} \rangle} && \cdots \\
	\\
	\cdots && {\mathbb{X}_i} && {\mathbb{X}_{i+1}} && {\mathcal{A}_{PL}(\mathbb{X}_{i+2})} && {\cdots }
	\arrow[from=1-3, to=1-1]
	\arrow[from=1-5, to=1-3]
	\arrow["{\langle H_i \rangle}", Rightarrow, from=1-5, to=3-3]
	\arrow[from=1-7, to=1-5]
	\arrow["{\langle H_{i+1} \rangle }", Rightarrow, from=1-7, to=3-5]
	\arrow[from=1-9, to=1-7]
	\arrow["\simeq"', from=3-3, to=1-3]
	\arrow[from=3-3, to=3-1]
	\arrow["\simeq"', from=3-5, to=1-5]
	\arrow[from=3-5, to=3-3]
	\arrow["\simeq"', from=3-7, to=1-7]
	\arrow[from=3-7, to=3-5]
	\arrow[from=3-9, to=3-7]
\end{tikzcd}\] This weak rational $2$-equivalence of copersistent spaces shows that 
$$\pi_*^\Q \langle \bM \rangle \cong \pi_*^\Q \X.$$ A direct application of the structure theorem (\ref{structure-theorem}) for persistent minimal model yields the following result generalizing (\ref{classical-postnikov}).

\begin{corollary} \label{persistent-eckmann-hilton}Let $\X : [n]^{op} \to \sSet$ be a copersistent, simply-connected space $\X$ with minimal model $\bM \to_2 \mathcal{A}_{PL}(X)$. The image of the skeletal filtration \ref{informal-cofiber-sequence-diagram} of $\bM$ under spatial realization 
    % https://q.uiver.app/#q=WzAsOSxbMCwwLCJcXGJpZ290aW1lcyBcXExhbWJkYSBcXG1hdGhiYntTfV57aysxfV97W3MsdCl9Il0sWzAsMiwiXFxiaWdvdGltZXMgXFxMYW1iZGEgXFxtYXRoYmJ7RH1ee2srMX1fe1tzLFxcaW5mdHkpfSJdLFsyLDAsIlxcbWF0aGJie019X3trLTF9Il0sWzIsMiwiXFxtYXRoYmJ7TX1fe2t9Il0sWzMsMSwiXFx4cmlnaHRhcnJvd3tcXGxhbmdsZSAtIFxccmFuZ2xlfSJdLFs0LDAsIlxcbGFuZ2xlIFxcbWF0aGJie019X3trfVxccmFuZ2xlIl0sWzYsMiwiSyhcXHBpX2teXFxRKFxcbWF0aGJie1h9KSwgaysxKSJdLFs2LDAsIlBLKFxccGlfa15cXFEoXFxtYXRoYmJ7WH0pLGsrMSkiXSxbNCwyLCJcXGxhbmdsZSBcXG1hdGhiYntNfV97ay0xfVxccmFuZ2xlIl0sWzAsMSwiIiwyLHsic3R5bGUiOnsidGFpbCI6eyJuYW1lIjoibW9ubyJ9fX1dLFswLDIsIlxcb3RpbWVzXFx0YXUiXSxbMSwzXSxbMiwzLCIiLDAseyJzdHlsZSI6eyJ0YWlsIjp7Im5hbWUiOiJtb25vIn19fV0sWzMsMCwiIiwxLHsic3R5bGUiOnsibmFtZSI6ImNvcm5lciJ9fV0sWzcsNiwiIiwyLHsic3R5bGUiOnsiaGVhZCI6eyJuYW1lIjoiZXBpIn19fV0sWzUsOCwiIiwwLHsic3R5bGUiOnsiaGVhZCI6eyJuYW1lIjoiZXBpIn19fV0sWzgsNiwiXFxwcm9kXFxsYW5nbGUgXFx0YXUgXFxyYW5nbGUiXSxbNSw3XSxbNSw2LCIiLDEseyJzdHlsZSI6eyJuYW1lIjoiY29ybmVyIn19XV0=
\[\begin{tikzcd}[sep=tiny]
	{\bigotimes \Lambda \mathbb{S}^{k+1}_{[s,t)}} && {\mathbb{M}_{k-1}} && {\langle \mathbb{M}_{k}\rangle} && {\prod PK(\bI_{[s,t)}^\vee,k+1)} \\
	&&& {\xrightarrow{\langle - \rangle}} \\
	{\bigotimes \Lambda \mathbb{D}^{k+1}_{[s,\infty)}} && {\mathbb{M}_{k}} && {\langle \mathbb{M}_{k-1}\rangle} && { \prod K(\bI_{[s,t)}^\vee,k+1)}
	\arrow["\otimes\tau", from=1-1, to=1-3]
	\arrow[tail, from=1-1, to=3-1]
	\arrow[tail, from=1-3, to=3-3]
	\arrow[from=1-5, to=1-7]
	\arrow[two heads, from=1-5, to=3-5]
	\arrow["\lrcorner"{anchor=center, pos=0.125}, draw=none, from=1-5, to=3-7]
	\arrow[two heads, from=1-7, to=3-7]
	\arrow[from=3-1, to=3-3]
	\arrow["\lrcorner"{anchor=center, pos=0.125, rotate=180}, draw=none, from=3-3, to=1-1]
	\arrow["{\prod\langle \tau \rangle}", from=3-5, to=3-7]
\end{tikzcd}\] recovers a copersistent Postnikov tower for $\X$.
\end{corollary} The construction on the right is the 'copersistent Postnikov tower' for (the rational homotopy type of) $\X$, and restricts to (\ref{classical-postnikov}) point-wise. The commutativity of products and spatial realization is a result of the (contravariant) left-adjointness. The decomposition of
$$ \langle \otimes \tau \rangle = \prod \Big( \langle \bM \rangle \xrightarrow{\langle \tau \rangle} K(\bI_{[s,t)}^\vee, k+1) \Big)$$
into products over the interval decomposition of $K(\pi_k^\Q(\X),k+1)$ reflects the basic correspondence:\\ 
$$\Big\{ \parbox{10em}{\centering attachments of\\ interval spheres} \Big\} \hspace{2em} \leftrightarrow \hspace{2em} \Big\{ \parbox{15em}{ \centering co-attachments of co-persistent\\ Eilenberg-Mac Lane intervals} \Big\} \vspace{1em}.$$ If all attaching maps in the minimal model are trivial, then
$$\X \simeq_2 \langle \bM \rangle \simeq \prod_k K( \pi_k^\Q(\X), k) \simeq \prod_\gamma K(\bI_{[s_\gamma,t_\gamma)},k_\gamma),$$ which we interpret as a `homotopical interval decomposition' of $\X$ at the level of copersistent spaces.

\begin{comment}
\subsection{Persistent $k$-invariants} With the previous example in mind, the attaching maps of the persistent minimal model --- or equivalently the classifying maps of the copersistent Postnikov tower --- encode information about the interval decomposability of copersistent spaces. 
\end{comment}

\subsection{Examples} The structural cell decomposition theorem \ref{structure-theorem} shows that the copersistent rational homotopy type of simply-connected $\X \in p^*\sSet$ is determined by 
\begin{enumerate}
    \item the intervals $\pi_*^\Q(\X)^\vee \cong \bigoplus \bI_{[s,t)}^k$, and
    \item their associated attaching maps in the minimal model $\bM$.
\end{enumerate} As the following examples show, the naturality of the coboundary map imposes algebraic constraints on the possible copersistent rational homotopy types associated to a given graded interval decomposition.\\

\subsubsection{Example I} For copersistent rational homotopy groups
% https://q.uiver.app/#q=WzAsMTEsWzAsMSwiXFxwaV5cXFFfKihcXGJYKV5cXHZlZT0iXSxbMSwxLCJcXGFscGhhX3AiXSxbMywxLCIwIl0sWzEsMiwicCJdLFsyLDIsInEiXSxbMywyLCJyIl0sWzQsMiwicyJdLFsyLDAsIlxcYmV0YV9xIl0sWzQsMCwiMCJdLFs1LDEsIjIiXSxbNSwwLCIzIl0sWzEsMl0sWzcsOF1d
\[\begin{tikzcd}[column sep=scriptsize,row sep=tiny]
	&& {\Q \cdot \beta_q} && 0 & 3 \\
	{\pi^\Q_*(\X)^\vee=} & {\Q \cdot \alpha_p} && 0 && 2, \\
	& p & q & r & s
	\arrow[from=1-3, to=1-5]
	\arrow[from=2-2, to=2-4]
\end{tikzcd}\] the $2$-minimal model is necessarily $\bM_2 = \Lambda \bI_{[p,r)}^2$. The two possible minimal models are determined by the coboundary of $\beta_q$, i.e., by the attaching maps
% https://q.uiver.app/#q=WzAsNSxbMCwwLCJcXFNwXns0fV97W3Escyl9Il0sWzIsMCwiXFxiTV8yIl0sWzMsMCwiXFx0ZXh0e29yfSJdLFs0LDAsIlxcU3BeezR9X3tbcSxzKX0iXSxbNiwwLCJcXGJNXzIiXSxbMCwxLCIoXFxhbHBoYV9xXjIsMCkiXSxbMyw0LCIwIl1d
\[\begin{tikzcd}[column sep=scriptsize,row sep=tiny]
	{\Sp^{4}_{[q,s)}} && {\bM_2} & {\text{or}} & {\Sp^{4}_{[q,s)}} && {\bM_2}
	\arrow["{(\alpha_q^2,0)}", from=1-1, to=1-3]
	\arrow["0", from=1-5, to=1-7]
\end{tikzcd}\] which give rise to the following minimal models
$$
\bM = p\Lambda( \alpha_p, \beta_q  \mid d \beta_q = \alpha_q^2) \hspace{1em} \text{or} \hspace{5em} p\Lambda( \alpha_p, \beta_q)$$ where we have omitted the trivial coboundary relations $d\alpha_q = 0$ and $d \beta_p=0$. These possibilities respectively model the tame copersistent spaces
% https://q.uiver.app/#q=WzAsMTEsWzAsMSwiXFxiWCA9Il0sWzEsMSwiSyhcXFEsMikiXSxbMiwxLCJTXjJfXFxRIl0sWzMsMSwiU14zX1xcUSJdLFs0LDEsIlxcdGV4dHtvcn0iXSxbNiwwLCJTXjNfXFxRIl0sWzcsMCwiU14zX1xcUSJdLFs2LDIsIksoXFxRLDIpIl0sWzUsMiwiSyhcXFEsMikiXSxbNiwxLCJcXHRpbWVzIl0sWzcsMiwiUEsoXFxRLDIpIl0sWzMsMiwiXFx0ZXh0e2hvcGZ9IiwyXSxbMiwxXSxbNiw1XSxbNyw4XSxbMTAsN11d
\[\begin{tikzcd}[column sep=scriptsize,row sep=tiny]
	&&&&&& {S^3_\Q} & {S^3_\Q} \\
	{\X \simeq_\Q} & {K(\Q,2)} & {S^2_\Q} & {S^3_\Q} & {\text{or}} && \times \\
	&&&&& {K(\Q,2)} & {K(\Q,2)} & {PK(\Q,2)}
	\arrow[from=1-8, to=1-7]
	\arrow[from=2-3, to=2-2]
	\arrow["{\text{hopf}}"', from=2-4, to=2-3]
	\arrow[from=3-7, to=3-6]
	\arrow[from=3-8, to=3-7]
\end{tikzcd}\]
up to copersistent rational equivalence. \\

\subsubsection{Example II} The commutativity of structure maps restricts the type of possible attaching maps. Making a minor adjustment to the previous example, let
% https://q.uiver.app/#q=WzAsMTEsWzIsMSwiXFxhbHBoYV9xIl0sWzQsMSwiMCJdLFsxLDIsInAiXSxbMiwyLCJxIl0sWzMsMiwiciJdLFs0LDIsInMiXSxbMSwwLCJcXGJldGFfcCJdLFszLDAsIjAiXSxbNSwxLCIyIl0sWzUsMCwiMyJdLFswLDEsIlxccGleXFxRXyooXFxiWCleXFx2ZWU9IFxcYkleM197W3Ascil9IFxcb3BsdXMgXFxiSV4yX3tbcSxzKX0gPSJdLFswLDFdLFs2LDddXQ==
\[\begin{tikzcd}[column sep=scriptsize,row sep=tiny]
	& {\Q \cdot \beta_p} && 0 && 3 \\
	{\pi^\Q_*(\X)^\vee= \bI^3_{[p,r)} \oplus \bI^2_{[q,s)} =} && {\Q \cdot \alpha_q} && 0 & 2 \\
	& p & q & r & s
	\arrow[from=1-2, to=1-4]
	\arrow[from=2-3, to=2-5]
\end{tikzcd}\] where the $2$-minimal model is now necessarily $\bM_2 = \Lambda \bI^2_{[q,s)}$. The only possible attaching map for persistent cell $\beta$ is trivial, implying $$\bM = p\Lambda( \alpha_q,\beta_p ) \hspace{4em} \text{and} \hspace{4em} \X \simeq_\Q K(\bI_{[q,s)},2) \times K(\bI_{[p,r)},3).$$\\

\subsubsection{Example III} In addition to the coboundary of generators, the presentation of a persistent minimal model has a second type of relation unique to the persistence setting: the mapping of the interval end-point into a non-trivial product of lower degree elements. For example, if 
% https://q.uiver.app/#q=WzAsMTEsWzIsMSwiXFxhbHBoYV9xIl0sWzQsMSwiMCJdLFsxLDIsInAiXSxbMiwyLCJxIl0sWzMsMiwiciJdLFs0LDIsInMiXSxbMSwwLCJcXGdhbW1hX3AiXSxbMywwLCIwIl0sWzUsMSwiMiJdLFs1LDAsIjQiXSxbMCwxLCJcXHBpXlxcUV8qKFxcYlgpXlxcdmVlPSBcXGJJXjRfe1twLHIpfSBcXG9wbHVzIFxcYkleMl97W3Escyl9ID0iXSxbMCwxXSxbNiw3XV0=
\[\begin{tikzcd}[column sep=scriptsize,row sep=tiny]
	& {\Q \cdot \gamma_p} && 0 && 4 \\
	{\pi^\Q_*(\X)^\vee = \bI^4_{[p,r)} \oplus \bI^2_{[q,s)} =} && {\Q \cdot \alpha_q} && 0 & 2 \\
	& p & q & r & s
	\arrow[from=1-2, to=1-4]
	\arrow[from=2-3, to=2-5]
\end{tikzcd}\] then $\bM_2 = \Lambda \bI_{[q,s)}^2$. The full minimal model is determined by the attaching map
% https://q.uiver.app/#q=WzAsNyxbMCwwLCIoMSkiXSxbMSwwLCJcXExhbWJkYSBcXFNwXjVfe1twLHEpfSJdLFszLDAsIlxcYk1fMiJdLFs1LDAsIlxcTGFtYmRhIFxcU3BeNV97W3AscSl9Il0sWzQsMCwiXFx0ZXh0e29yfSJdLFs3LDAsIlxcYk1fMiJdLFs4LDAsIigyKSJdLFsxLDIsIigwLFxcYWxwaGFfcl4yKSJdLFszLDUsIjAiXV0=
\[\begin{tikzcd}[column sep=scriptsize,row sep=tiny]
	{(1)} & {\Lambda \Sp^5_{[p,q)}} && {\bM_2} & {\text{or}} & {\Lambda \Sp^5_{[p,q)}} && {\bM_2} & {(2)}
	\arrow["{(0,\alpha_r^2)}", from=1-2, to=1-4]
	\arrow["0", from=1-6, to=1-8]
\end{tikzcd}\] of the persistent cell $\gamma$ with respective presentations
$$\bM = p\Lambda(\gamma_p, \alpha_q \mid \gamma_q = \alpha_q^2) \hspace{3em} \text{or} \hspace{3em} p\Lambda(\gamma_p, \alpha_q). \hspace{6em}$$ At the level of spaces these correspond to
% https://q.uiver.app/#q=WzAsMyxbMCwwLCJLKFxcUSw0KSJdLFsyLDAsIksoXFxRLDQpXFx0aW1lcyBLKFxcUSwyKSJdLFs0LDAsIksoXFxRLDIpIl0sWzEsMF0sWzIsMSwiZiIsMl1d
\[\begin{tikzcd}[column sep=scriptsize,row sep=tiny]
	{K(\Q,4)} && {K(\Q,4)\times K(\Q,2)} && {K(\Q,2)}
	\arrow["\pi_{K(\Q,4)}"', from=1-3, to=1-1]
	\arrow["f"', from=1-5, to=1-3]
\end{tikzcd}\]
where $f$ is distinguished in the two cases by the induced map on the free polynomial cohomology rings
$$H(f) : \Q[\gamma,\alpha]  \to \Q[\alpha]; \hspace{5em} \alpha  \mapsto \alpha \hspace{5em} \gamma  \mapsto \begin{cases} \alpha^2 & (1) \\ 0 & (2). \end{cases} $$

\bibliography{misc/bibliography}
\bibliographystyle{alpha} 
\end{document}